\documentclass[twoside,a4paper,reqno,11pt]{amsart}

\usepackage[margin=1in]{geometry}
\usepackage{amsmath,amsthm,amsfonts,amssymb,latexsym, mathrsfs}

\usepackage{hyperref}

\usepackage{enumerate}
\usepackage{color}
\usepackage[dvipsnames]{xcolor}

\usepackage[shortlabels]{enumitem}
\usepackage{comment}

\begin{document}

\theoremstyle{plain}

\newtheorem{thm}{Theorem}[section]

\theoremstyle{plain}
\newtheorem{theoA}{Theorem}

\renewcommand{\thetheoA}{\Alph{theoA}}

\newtheorem{lem}[thm]{Lemma}
\newtheorem{Problem B}[thm]{Problem B}

\newtheorem{pro}[thm]{Proposition}
\newtheorem{conj}[thm]{Conjecture}
\newtheorem{cor}[thm]{Corollary}
\newtheorem{que}[thm]{Question}
\newtheorem{prob}[thm]{Problem}
\newtheorem{rem}[thm]{Remark}
\newtheorem{defi}[thm]{Definition}
\newtheorem{cond}[thm]{Condition}

\newtheorem*{thmA}{Theorem A}
\newtheorem*{thmB}{Theorem B}
\newtheorem*{corB}{Corollary B}
\newtheorem*{thmC}{Theorem C}
\newtheorem*{thmD}{Theorem D}
\newtheorem*{thmE}{Theorem E}
 
\newtheorem*{thmAcl}{Main Theorem$^{*}$}
\newtheorem*{thmBcl}{Theorem B$^{*}$}
\newcommand{\dd}{\mathrm{d}}

\newtheorem{thml}{Theorem}
\renewcommand*{\thethml}{\Alph{thml}}   
\newtheorem{conjl}[thml]{Conjecture}
\newtheorem{condl}[thml]{Condition}

\newcommand{\Maxn}{\operatorname{Max_{\textbf{N}}}}
\newcommand{\Syl}{\operatorname{Syl}}
\newcommand{\Lin}{\operatorname{Lin}}
\newcommand{\U}{\mathbf{U}}
\newcommand{\R}{\mathbf{R}}
\newcommand{\dl}{\operatorname{dl}}
\newcommand{\Con}{\operatorname{Con}}
\newcommand{\cl}{\operatorname{cl}}
\newcommand{\Stab}{\operatorname{Stab}}
\newcommand{\Aut}{\operatorname{Aut}}
\newcommand{\Ker}{\operatorname{Ker}}
\newcommand{\InnDiag}{\operatorname{InnDiag}}
\newcommand{\fl}{\operatorname{fl}}
\newcommand{\Irr}{\operatorname{Irr}}
\newcommand{\FF}{\mathbb{F}}
\newcommand{\EE}{\mathbb{E}}
\newcommand{\normal}{\trianglelefteq}
\newcommand{\sn}{\normal\normal}
\newcommand{\Bl}{\mathrm{Bl}}
\newcommand{\NN}{\mathbb{N}}
\newcommand{\N}{\mathbf{N}}
\newcommand{\bfC}{\mathbf{C}}
\newcommand{\bfO}{\mathbf{O}}
\newcommand{\bfF}{\mathbf{F}}
\def\GGG{{\mathcal G}}
\def\HHH{{\mathcal H}}
\def\HH{{\mathcal H}}
\def\irra#1#2{{\rm Irr}_{#1}(#2)}

\renewcommand{\labelenumi}{\upshape (\roman{enumi})}

\newcommand{\PSL}{\operatorname{PSL}}
\newcommand{\PSU}{\operatorname{PSU}}
\newcommand{\alt}{\operatorname{Alt}}
\newcommand{\miquelcomment}{\textcolor{blue}}

\providecommand{\V}{\mathrm{V}}
\providecommand{\E}{\mathrm{E}}
\providecommand{\ir}{\mathrm{Irm_{rv}}}
\providecommand{\Irrr}{\mathrm{Irm_{rv}}}
\providecommand{\re}{\mathrm{Re}}

\numberwithin{equation}{section}
\def\irrp#1{{\rm Irr}_{p'}(#1)}

\def\ibrrp#1{{\rm IBr}_{\Bbb R, p'}(#1)}
\def\C{{\mathbb C}}
\def\Q{{\mathbb Q}}
\def\irr#1{{\rm Irr}(#1)}
\def\irrp#1{{\rm Irr}_{p^\prime}(#1)}
\def\irrq#1{{\rm Irr}_{q^\prime}(#1)}
\def \c#1{{\cal #1}}
\def \aut#1{{\rm Aut}(#1)}
\def\cent#1#2{{\bf C}_{#1}(#2)}
\def\norm#1#2{{\bf N}_{#1}(#2)}
\def\zent#1{{\bf Z}(#1)}
\def\syl#1#2{{\rm Syl}_#1(#2)}
\def\normal{\triangleleft\,}
\def\oh#1#2{{\bf O}_{#1}(#2)}
\def\Oh#1#2{{\bf O}^{#1}(#2)}
\def\det#1{{\rm det}(#1)}
\def\gal#1{{\rm Gal}(#1)}
\def\ker#1{{\rm ker}(#1)}
\def\normalm#1#2{{\bf N}_{#1}(#2)}
\def\alt#1{{\rm Alt}(#1)}
\def\iitem#1{\goodbreak\par\noindent{\bf #1}}
   \def \mod#1{\, {\rm mod} \, #1 \, }
\def\sbs{\subseteq}

\def\gc{{\bf GC}}
\def\ngc{{non-{\bf GC}}}
\def\ngcs{{non-{\bf GC}$^*$}}
\newcommand{\notd}{{\!\not{|}}}

\newcommand{\Z}{\mathbf{Z}}
\newcommand{\Out}{{\mathrm {Out}}}
\newcommand{\Mult}{{\mathrm {Mult}}}
\newcommand{\Inn}{{\mathrm {Inn}}}
\newcommand{\IBR}{{\mathrm {IBr}}}
\newcommand{\IBRL}{{\mathrm {IBr}}_{\ell}}
\newcommand{\IBRP}{{\mathrm {IBr}}_{p}}
\newcommand{\cd}{\mathrm{cd}}
\newcommand{\ord}{{\mathrm {ord}}}
\def\id{\mathop{\mathrm{ id}}\nolimits}
\renewcommand{\Im}{{\mathrm {Im}}}
\newcommand{\Ind}{{\mathrm {Ind}}}
\newcommand{\diag}{{\mathrm {diag}}}
\newcommand{\soc}{{\mathrm {soc}}}
\newcommand{\End}{{\mathrm {End}}}
\newcommand{\sol}{{\mathrm {sol}}}
\newcommand{\Hom}{{\mathrm {Hom}}}
\newcommand{\Mor}{{\mathrm {Mor}}}
\newcommand{\Mat}{{\mathrm {Mat}}}
\def\rank{\mathop{\mathrm{ rank}}\nolimits}
\newcommand{\Tr}{{\mathrm {Tr}}}
\newcommand{\tr}{{\mathrm {tr}}}
\newcommand{\Gal}{{\rm Gal}}
\newcommand{\Spec}{{\mathrm {Spec}}}
\newcommand{\ad}{{\mathrm {ad}}}
\newcommand{\Sym}{{\mathrm {Sym}}}
\newcommand{\Char}{{\mathrm {Char}}}
\newcommand{\pr}{{\mathrm {pr}}}
\newcommand{\rad}{{\mathrm {rad}}}
\newcommand{\abel}{{\mathrm {abel}}}
\newcommand{\PGL}{{\mathrm {PGL}}}
\newcommand{\PCSp}{{\mathrm {PCSp}}}
\newcommand{\PGU}{{\mathrm {PGU}}}
\newcommand{\codim}{{\mathrm {codim}}}
\newcommand{\ind}{{\mathrm {ind}}}
\newcommand{\Res}{{\mathrm {Res}}}
\newcommand{\Lie}{{\mathrm {Lie}}}
\newcommand{\Ext}{{\mathrm {Ext}}}
\newcommand{\Alt}{{\mathrm {Alt}}}
\newcommand{\AAA}{{\sf A}}
\newcommand{\SSS}{{\sf S}}
\newcommand{\DDD}{{\sf D}}
\newcommand{\QQQ}{{\sf Q}}
\newcommand{\CCC}{{\sf C}}
\newcommand{\SL}{{\mathrm {SL}}}
\newcommand{\Sp}{{\mathrm {Sp}}}
\newcommand{\PSp}{{\mathrm {PSp}}}
\newcommand{\SU}{{\mathrm {SU}}}
\newcommand{\GL}{{\mathrm {GL}}}
\newcommand{\GU}{{\mathrm {GU}}}
\newcommand{\Spin}{{\mathrm {Spin}}}
\newcommand{\CC}{{\mathbb C}}
\newcommand{\CB}{{\mathbf C}}
\newcommand{\RR}{{\mathbb R}}
\newcommand{\QQ}{{\mathbb Q}}
\newcommand{\ZZ}{{\mathbb Z}}
\newcommand{\bfN}{{\mathbf N}}
\newcommand{\bfZ}{{\mathbf Z}}
\newcommand{\PP}{{\mathbb P}}
\newcommand{\cG}{{\mathcal G}}
\newcommand{\OO}{\mathcal O}
\newcommand{\cH}{{\mathcal H}}
\newcommand{\cQ}{{\mathcal Q}}
\newcommand{\GA}{{\mathfrak G}}
\newcommand{\cT}{{\mathcal T}}
\newcommand{\cL}{{\mathcal L}}
\newcommand{\IBr}{\mathrm{IBr}}
\newcommand{\cS}{{\mathcal S}}
\newcommand{\cR}{{\mathcal R}}
\newcommand{\GCD}{\GC^{*}}
\newcommand{\TCD}{\TC^{*}}
\newcommand{\FD}{F^{*}}
\newcommand{\GD}{G^{*}}
\newcommand{\HD}{H^{*}}
\newcommand{\GCF}{\GC^{F}}
\newcommand{\bl}{\mathrm{bl}}
\newcommand{\TCF}{\TC^{F}}
\newcommand{\PCF}{\PC^{F}}
\newcommand{\GCDF}{(\GC^{*})^{F^{*}}}
\newcommand{\RGTT}{R^{\GC}_{\TC}(\theta)}
\newcommand{\RGTA}{R^{\GC}_{\TC}(1)}
\newcommand{\Om}{\Omega}
\newcommand{\eps}{\epsilon}
\newcommand{\varep}{\varepsilon}
\newcommand{\al}{\alpha}
\newcommand{\chis}{\chi_{s}}
\newcommand{\sigmad}{\sigma^{*}}
\newcommand{\PA}{\boldsymbol{\alpha}}
\newcommand{\gam}{\gamma}
\newcommand{\lam}{\lambda}
\newcommand{\la}{\langle}
\newcommand{\genf}{F^*}
\newcommand{\ksigma}{k_{0,\sigma}}
\newcommand{\irrsigma}{\Irr_{0,\sigma}}
\newcommand{\irrtau}{\Irr_{0,\tau}}
\newcommand{\irrpsigma}{\Irr_{0,P,\sigma}}
\newcommand{\ra}{\rangle}
\newcommand{\hs}{\hat{s}}
\newcommand{\htt}{\hat{t}}
\newcommand{\tG}{\hat G}
\newcommand{\St}{\mathsf {St}}
\newcommand{\bfs}{\boldsymbol{s}}
\newcommand{\bfl}{\boldsymbol{\lambda}}
\newcommand{\tn}{\hspace{0.5mm}^{t}\hspace*{-0.2mm}}
\newcommand{\ta}{\hspace{0.5mm}^{2}\hspace*{-0.2mm}}
\newcommand{\tb}{\hspace{0.5mm}^{3}\hspace*{-0.2mm}}
\def\skipa{\vspace{-1.5mm} & \vspace{-1.5mm} & \vspace{-1.5mm}\\}
\newcommand{\tw}[1]{{}^#1\!}
\newcommand{\Irrg}[1]{\Irr_{p',\sigma}(#1)}
\renewcommand{\mod}{\bmod \,}

\newcommand{\bG}{\mathbf{G}}
\newcommand{\bbG}{\mathbb{G}}
\newcommand{\bg}{\mathbf}
\newcommand{\type}{\operatorname}
\newcommand{\wt}{\widetilde}
\newcommand{\sym}{\mathfrak{S}}

\marginparsep-0.5cm

\newcommand{\mandicomment}{\textcolor{teal}}

\subjclass[2020]{20C15, 20D20}
\keywords{Galois action, principal block, Sylow subgroup}

\title[Galois action and generation of Sylow $3$-subgroups]{Galois action on the principal block and generation of Sylow 3-subgroups}

\author[E. Ketchum]{Eden Ketchum}
\address[Ketchum, Schaeffer Fry]{Department of Mathematics, University of Denver, Denver, CO 80210,
USA}

\author[J. M. Mart\'inez]{J. Miquel Mart\'inez}
 \address[Mart\'inez, Rizo]{Departament de Matem\`atiques, Universitat de Val\`encia, 46100 Burjassot,
 Val\`encia, Spain}

\author[N. Rizo]{Noelia Rizo}

\author[A. A Schaeffer Fry]{A. A. Schaeffer Fry}
\email[Ketchum]{eden.ketchum@du.edu}
\email[Mart\'inez]{josep.m.martinez@uv.es}
\email[Rizo]{noelia.rizo@uv.es}
\email[Schaeffer Fry]{mandi.schaefferfry@du.edu}

\thanks{The first and fourth-named authors acknowledge support of a CAREER grant from the U.S. National Science Foundation, Award No. DMS-2439897; the second and third-named authors acknowledge support from grant
PID2022-137612NB-I00 funded by MCIN/AEI/ 10.13039/501100011033 and ERDF “A way of making Europe” and grant CIDEIG/2022/29 funded by Generalitat Valenciana. Part of this work was done while the second and third-named authors visited the University of Denver. They wish to thank the Department of Mathematics for their kind hospitality.}

\begin{abstract}

 In this paper, we prove one direction of a conjecture of Navarro--Rizo--Schaeffer Fry--Vallejo positing an algorithm to determine from the character table whether a finite group has $2$-generated Sylow $3$-subgroups. This gives further evidence of the blockwise version of the Galois--McKay conjecture (also known as the Alperin--McKay--Navarro conjecture). A key step involves proving the Isaacs--Navarro Galois conjecture for principal blocks for finite groups with a certain structure.
\end{abstract}
\maketitle

\section{Introduction}

In Problem 12 of his famous list of problems from 1963, Brauer asked what information could be gleaned about a Sylow $p$-subgroup $P$ of a finite group $G$ from its character table. In particular, the question of whether the number of generators of $P$ could be obtained in this way has been studied in recent years, see e.g. \cite{Riz-Sch-Val20, Nav-Riz-Sch-Val20, Mor-Sam23, Va23}. Many questions in this realm seem linked to the action of the Galois automorphism $\sigma\in\Gal(\mathbb{Q}^{ab}/\mathbb{Q})$ that fixes $p'$-roots of unity and sends $p$-power roots of unity to their $1+p$ power. For any cyclotomic extension $\mathbb{Q}_n:=\mathbb{Q}(e^{2\pi i/n})$ of $\mathbb{Q}$, we keep the notation $\sigma$ for the restriction of this automorphism to $\mathbb{Q}_n$.
Here, we contribute to this line of problems by proving the following direction of the main conjecture from \cite{Nav-Riz-Sch-Val20}.
\begin{theoA}\label{thm:A}
    Let $G$ be a finite group and $P\in\Syl_3(G)$. Then $|P:\Phi(P)|=9$ if the principal $3$-block $B_0(G)$ contains exactly $6$ or $9$ $\sigma$-invariant characters of degree coprime to $3$.
\end{theoA}

As noted in \cite{Nav-Riz-Sch-Val20}, Theorem \ref{thm:A} (and its converse) would follow from the Alperin--McKay--Navarro conjecture \cite[Conj.~B]{Nav04}, or even the more restrictive version  \cite[Conj.~D]{Isa-Nav02}. Hence, this result gives further evidence for these elusive conjectures. We remark that although the block-free version \cite[Conj.~C]{Isa-Nav02} was recently established in \cite{RSF25}, the blockwise version (even for principal blocks) is much further from completion and has not yet been reduced to a problem on simple groups. For this reason, finding additional evidence of this blockwise version remains pertinent.  

Theorem \ref{thm:A} can be thought of simultaneously as an extension of the main result of \cite{Nav-Riz-Sch-Val20}, which addressed the analogous question for $p=2$, and of the main result of \cite{Gia-Riz-Sch-Val24}, which showed that $[P:P']=9$ if and only if $B_0(G)$ contains exactly 6 or 9 irreducible characters of degree coprime to $3$.

Our proof of Theorem \ref{thm:A} uses the recent work of Ketchum \cite{eden}, which shows that the statement of Theorem \ref{thm:A} (and its converse) hold for almost simple groups. It also uses a reduction theorem and additional results on almost simple groups analogous to those in \cite{Gia-Riz-Sch-Val24}. Moreover, the reduction relies on beautiful work of R. Brauer \cite{Bra76} on the structure of groups having cyclic Sylow $p$-subgroups.

As mentioned above, \cite[Conj.~C]{Isa-Nav02} has been completed in \cite{RSF25}. Our next main theorem considers the principal block version, which is part of \cite[Conj.~D]{Isa-Nav02}, and is a key step towards the proof of Theorem \ref{thm:A}. We believe it may be of independent interest. Let $\mathcal{H}\leq \Gal(\mathbb{Q}^{ab}/\mathbb{Q})$ be the Galois group considered in \cite{Nav04}.

\begin{theoA}\label{thm:B}
Let $G$ be a finite group of order divisible by $p$, let $N\normal G$ be an abelian $p$-group and assume $G/N$ has cyclic Sylow $p$-subgroups and let $\tau\in\mathcal{H}$ be a $p$-power order element. Then $\tau$ fixes the same number of height-zero characters in $\Irr(B_0(G))$ as in $\Irr(B_0(\norm G P))$, where $P\in{\rm Syl}_p(G)$.
\end{theoA}

In particular, \cite[Conj.~D]{Isa-Nav02} holds for principal blocks of groups satisfying the hypotheses of Theorem \ref{thm:B}. In fact, as explained in Remark \ref{rem:thm B}, the same result holds by assuming only that every block of maximal defect of $G/N$ satisfies \cite[Conj.~D]{Isa-Nav02}, or \cite[Conj.~B]{Nav04}. (See also \cite[Prop.~5.4]{Mar-Mar-Sch-Val24} for another related, but restrictive, case.)

The paper is structured as follows: Section \ref{sec:aux} contains preliminary lemmas and results, as well as a Galois version of \cite[Thm.~B]{Gia-Riz-Sch-Val24}. Theorem \ref{thm:B} is proved in Section \ref{sec:thm B}. Theorem \ref{thm:A} is reduced to simple groups in Section \ref{sec:reduction} and finally completed in Section \ref{sec:simples}, where we also provide some evidence for the version of Theorem \ref{thm:A} for arbitrary blocks.

\section{Auxiliary results}\label{sec:aux}

\subsection{Galois action and principal blocks}

We begin by collecting some basic facts about principal blocks. Throughout, given a prime $p$, $B_0(G)$ denotes the principal $p$-block of the finite group $G$. We denote by $\Irr(B_0(G))$ the set of irreducible characters in $B_0(G)$ and by $\Irr_0(B_0(G))$ the subset of those with height zero (that is, $p'$-degree). The  set of $\sigma$-invariant characters in $\Irr_0(B_0(G))$ will be denoted $\irrsigma(B_0(G))$, and we set $\ksigma(B_0(G)):=|\irrsigma(B_0(G))|$.

\begin{lem}\label{lem:princblockabove} Let $G$ be a finite group, let $N\normal G$ and let $p$ be a prime. 
 \begin{enumerate}
 \item $\irr {B_0(G/N)}\sbs \irr{B_0(G)}$, and if $N$ has order not divisible by $p$ then $\irr {B_0(G/N)}= \irr{B_0(G)}$.
 \item For any $\theta\in\irr{B_0(N)}$, there is some $\chi\in\irr{B_0(G})$ lying over $\theta$.
  \item For any  $\chi\in\irr{B_0(G)}$, every constituent of $\chi_N$ lies in $B_0(N)$.
  \item If $G/N$ is a $p$-group, then $B_0(G)$ is the unique block covering $B_0(N)$.
  \item Assume $G=H_1\times\dots\times H_t$. Then 
$$\Irr(B_0(G))=\{\theta_1\times\dots\times\theta_t\mid\theta_i\in\Irr(B_0(H_i))\}.$$
 \end{enumerate}
 \end{lem}
 \begin{proof}
The first part of (i) follows from the discussion before \cite[Thm.~7.6]{Nav98}, and the second part is \cite[Thm.~9.9(c)]{Nav98}. Part (ii) is \cite[Thm.~9.4]{Nav98},  part (iii) is \cite[Thm.~9.2, Cor.~9.3]{Nav98} and part (iv) is \cite[Cor.~9.6]{Nav98}. Part (v) is \cite[Lem.~2.6(b)]{Nav-Tie16}.
 \end{proof}

 \begin{lem}\label{lem:carolina}
     Let $G$ be a finite group and let $N=S_1\times\cdots\times S_t$ be a minimal normal subgroup of $G$, where the $S_i$ are transitively permuted by $G$. Suppose that $G/N=\langle xN\rangle$ is cyclic, $\cent G N=1$ and $\norm G {S_i}=N$. Then $G/N\lesssim  \Out(N)=\Out(S_1)\wr {\sf S}_t$. In particular, after a suitable reordering of the $S_i$'s, $xN$ acts on $N$ as $(\alpha_1,\dots,\alpha_t)\tau\in \Out(S_1)\wr \mathsf{S}_t$ where $\tau=(1...t)$, and we have
     \begin{enumerate}
         \item $\alpha_1\cdots\alpha_t=1$;
         \item if $\eta\in\Irr(S_1)$, then the character
$\theta=\eta\times\eta^{\alpha_1}\times\eta^{\alpha_1\alpha_{2}}\times\cdots\times\eta^{\alpha_1\cdots\alpha_{t}}$
         is $G$-invariant.
         
     \end{enumerate}

 \end{lem}

 \begin{proof}
     Since $G$ acts transitively on the components of $N$, $G/N\leq\Out(N)\cong\Out(S_1)\wr\mathsf{S}_t$.
  Thus, $xN$ seen as an element of ${\rm Out}(N)$ acts on characters of $N$ like some $(\alpha_1, \alpha_2, \ldots, \alpha_r)\tau$ with $\alpha_i \in {\rm Out}(S)$ and $\tau \in {\sf S}_t$ of order $t$ and after a suitable reordering of the elements we may assume $\tau=(1...t)$. 
  Moreover, since $\norm G {S_1}=N$, then $o(xN)=|G:N|=t$ and we have that 
  \[1=((\alpha_1,\dots,\alpha_t)\tau)^t=(\prod_{j=0}^{t-1}\alpha_{\tau^j(1)},\dots,\prod_{j=0}^{t-1}\alpha_{\tau^j(t)})\tau^t.\] 
  Therefore, $$\prod_{j=0}^{t-1}\alpha_{\tau^j(1)}=\alpha_1\cdots\alpha_t=1,$$ proving part (i).
  Now, if $\eta\in\irr S$, then let $\theta=\eta\times\eta^{\alpha_1}\times\eta^{\alpha_1\alpha_{2}}\times\cdots\times\eta^{\alpha_1\cdots\alpha_{t}}$ and notice that $\theta^x=\theta$ by part (ii). 
  Since $G/N=\langle xN\rangle$, $\theta$ is $G$-invariant.
 \end{proof}

 \begin{thm}[Alperin--Dade]\label{thm:Alperin-Dade}
 Let $N\normal G$ and $P\in\Syl_p(G)$. If $p$ does not divide $|G:N|$ and $N\cent G P=G$, then restriction defines a bijection
 $$\irrsigma(B_0(G))\rightarrow\irrsigma(B_0(N)).$$
 \end{thm}
 \begin{proof}
     The fact that restriction defines a bijection
     $\Irr(B_0(G))\rightarrow\Irr(B_0(N))$ was proved when $G/N$ is solvable in \cite{alperin76} and without the solvability hypothesis in \cite{dade77}. Since $\chi_N^\sigma=(\chi^\sigma)_N$, the result follows.
 \end{proof}

\begin{lem}\label{lem:3.7MMSV}
Let $G$ be a finite group $H\leq G$ with $p\nmid |G:H|$. Let $\tau\in\Gal(\mathbb{Q}_{|G|}/\mathbb{Q})$ have $p$-power order. Assume that $\cent G Q\sbs H$ for $Q\in\Syl_p(H)$. If $\theta\in\irrtau(B_0(H))$ then $\theta^G$ contains a constituent $\chi\in\irrtau(B_0(G))$.
\end{lem}
\begin{proof}
This follows the proof of \cite[Lem.~3.7]{Mar-Mar-Sch-Val24} using that $\langle \tau\rangle$ has $p$-power order.
\end{proof}

\begin{lem}\label{lem:linear sigma-invariant}
    Let $N\normal G$ be a $p$-subgroup contained in $P\in\Syl_p(G)$. Further, assume $\theta\in\Lin(N)$ is $G$-invariant.  Let $\tau\in\Gal(\mathbb{Q}_{|G|}/\mathbb{Q})$ have $p$-power order and assume $\theta$ extends to $\lambda\in\irrtau(P)$. Then $\theta$ extends to a $\langle\tau\rangle$-invariant character in $B_0(G)$.
\end{lem}
\begin{proof}
Let $H=P\cent G P$ and write $H=P\times X$.  Letting $\hat{\lambda}=\lambda\times 1_X\in\Irr_{0,\tau}(B_0(H))$, we have that $\hat\lambda^G$ contains some $\chi\in\irrtau(B_0(G))$ by Lemma \ref{lem:3.7MMSV}. Now $\det{\chi^\tau}=\det{\chi}^\tau$ by \cite[Prob.~3.8]{Nav18}, so $\delta:=\det{\chi}^b$ (where $b$ is as in the proof of \cite[Lem.~2.5]{Gia-Riz-Sch-Val24}) is $\langle \tau\rangle$-invariant, lies in $B_0(G)$ and extends $\theta$.
\end{proof}

The following is a principal block version of \cite[Lem.~3]{Mor-Sam23}.

\begin{lem}\label{lem:kernel of sigmas}
Let $G$ be a finite group with $\oh{p'}G=1$ and let $P\in\Syl_p(G)$. Then
$$K=\bigcap_{\chi\in\irrsigma(B_0(G))}\ker\chi\leq \Phi(P).$$
\end{lem}
\begin{proof}
 Let $\lambda\in\Irr(P/\Phi(P))$. Write $P\cent G P=P\times X$ and let $\psi=\lambda\times 1_X\in\irrsigma(B_0(P\cent G P))$. By Lemma \ref{lem:3.7MMSV} we have that $\lambda^G$ contains a constituent $\chi\in\irrsigma(B_0(G))$. By Frobenius reciprocity, $\chi_{P\cent G P}$ contains $\lambda\times 1_X$ so $\chi_P$ contains $\lambda_P$. Moreover $\chi_{P\cap K}=\chi(1)1_{P\cap K}$ contains $\lambda_{P\cap K}$ which shows that $\ker{\lambda}$ contains $P\cap K$. Therefore $$P\cap K\sbs\bigcap_{\lambda\in\Irr(P/\Phi(P))}\ker\lambda=\Phi(P).$$
By Tate's theorem \cite[Cor.~6.14]{Nav18}, this implies that $K$ is $p$-nilpotent. Then if $K$ is not a $p$-group, $\oh{p'}K>1$, which contradicts $\oh{p'}G=1$. Therefore $K$ is a normal $p$-subgroup so $K=P\cap K\sbs\Phi(P)$, as desired.
\end{proof}

The following is a variation of \cite[Thm.~3.5]{Gia-Riz-Sch-Val24}.

\begin{lem}\label{lem:tensor induction sigma}
    Assume $N\normal G$ is a direct product $N=S_1\times \dots\times S_t$ of nonabelian simple groups of order divisible by $p$ transitively permuted by $G$. Let $\theta=\theta_1\times\dots\times \theta_t$ be $G$-invariant. If $\theta_1$ extends to some $\sigma$-invariant irreducible character of $p'$-degree in $B_0(T)$ for all $S_1\leq T\leq\Aut(S_1)_{\theta_1}$, then $\theta$ extends to some $\chi\in\irrsigma(B_0(G))$
\end{lem}
\begin{proof}
    Following the proof of \cite[Thm.~3.5]{Gia-Riz-Sch-Val24} we have that $\theta_1$ extends to some $\hat\theta_1$ in $\irrsigma(B_0(\norm G {S_1}))$ by hypothesis. Let $\chi$ be the tensor induced character $\chi=\hat\theta_1^\otimes$ (see \cite[Sec.~10.5]{Nav18}), which lies in $B_0(G)$ and extends $\theta$ (again by \cite[Thm.~3.5]{Gia-Riz-Sch-Val24}). By the formula in \cite[Def.~2.1]{Glu-Isa83}, $\mathbb{Q}(\chi)\sbs\mathbb{Q}(\hat\theta_1)$ and it follows that $\chi$ is $\sigma$-invariant, as desired.
\end{proof}

The following are the only results of this section that are not valid for arbitrary primes. Next is a $\sigma$-version of \cite[Thm.~B]{Gia-Riz-Sch-Val24}.

\begin{pro}\label{pro:noelia relative sigma} Let $G$ be a finite group, $p\in\{2,3\}$, and $P\in\Syl_p(G)$. Let $N\normal G$ and assume $p$ divides $|G:N|$. Suppose that $\theta\in\Irr(B_0(N))$ is $P\times\langle \sigma\rangle$-invariant. Then $p$ divides $ \ksigma(B_0(G)|\theta)$. In particular, if $\theta$ extends to $\hat\theta\in\irrsigma(B_0(PN))$ then $\ksigma(B_0(G)|\theta)\geq p$.
\end{pro}
\begin{proof}
    Arguing as in \cite[Thm.~2.7]{Gia-Riz-Sch-Val24} we get that
    $$\sum_{\chi\in\Irr_0(B_0(G)|\theta)}\chi(1)^2\equiv 0\mod p.$$

    Recall that $\langle \sigma\rangle$ is a $p$-group and acts on $\Irr_0(B_0(G)|\theta)$. Write
    $$\Irr_0(B_0(G)|\theta)=\irrsigma(B_0(G)|\theta)\cup\Omega_1\cup\dots\cup\Omega_t$$
    where the $\Omega_i$'s are the nontrivial $\langle \sigma\rangle$-orbits, and let $\chi_i$ be a representative of each $\Omega_i$. Since each element of $\Omega_i$ has degree $\chi_i(1)$ and $|\Omega_i|\equiv 0\mod p$, we have
    $$\sum_{\chi\in\Irr_0(B_0(G)|\theta)}\chi(1)^2=\sum_{\chi\in\irrsigma(B_0(G)|\theta)}\chi(1)^2+\sum_{i=1}^t|\Omega_i|\chi_i(1)^2\equiv\sum_{\chi\in\irrsigma(B_0(G)|\theta)}\chi(1)^2 \mod p$$
and this shows that 
$$\sum_{\chi\in\irrsigma(B_0(G)|\theta)}\chi(1)^2 \equiv 0 \mod p.$$
Now, for every $\chi\in\irrsigma(B_0(G)|\theta)$ we have $\chi(1)^2\equiv 1 \mod p$ and we conclude that $p$ divides
$\ksigma(B_0(G)|\theta)$, as desired.

For the final part, notice that if there is some $\hat\theta\in\irrsigma(B_0(PN)|\theta)$ extending $\theta$ then there is $\psi\in\irrsigma(B_0(PN\cent G P)|\theta)$ extending $\hat\theta$ by Theorem \ref{thm:Alperin-Dade}. By Lemma \ref{lem:3.7MMSV}, $\psi^G$ contains some $\chi\in\irrsigma(B_0(G))$. Therefore $\ksigma(B_0(G)|\theta)>0$ by Frobenius reciprocity, so the result follows.
\end{proof}

\begin{rem}
The previous result is also true for non-principal blocks of maximal defect, where one has to use that they contain at least one $\sigma$-invariant height-zero character (this was proved in \cite[Prop.~2.6]{Bro-Pui80I}). 
\end{rem}

\begin{lem}\label{lem:p-action}
Assume $p\leq 3$ and that a $p$-group $P$ acts  by automorphisms on a group $K$ of order divisible by $p$. Then the set of $P$-invariant characters in $\irrsigma(B_0(K))$ is nonempty and has size divisible by $p$.
\end{lem}
\begin{proof}
This follows from \cite[Lem.~2.2(c)]{Riz-Sch-Val20} using that $P$ fixes $1_K\in\irrsigma(B_0(K))$.
\end{proof}

\subsection{Structure of groups with a cyclic Sylow $3$-subgroup}

The following result of Brauer (which relies on a result of Herzog in \cite{Her70}) will be essential in the final step of our reduction theorem for Theorem A. We restate it here for the reader's convenience.

\begin{thm}\label{Bra76} Let $G$ be a finite group, let $p=3$ and suppose that $G$ has cyclic Sylow $p$-subgroups. Then one of the following holds:
\begin{enumerate}
    \item $G$ is $p$-solvable and there exists $L\leq G$ with $|G:L|\in\{1,2\}$ such that $L$ has a normal $p$-complement.
    \item $G$ is not $p$-solvable and for every $N\lhd G$, either $N$ or $G/N$ is of order not divisible by $p$.
\end{enumerate}
\end{thm}
\begin{proof}
    We use the notation of \cite[Sec.~2]{Bra76}. Let $P$ be a Sylow 3-subgroup of $G$ and let $r=|\norm G P:P\cent G P|$. Since $R=\norm G P/P\cent G P\leq{\rm Aut}(P)$, we have that $r=|R|$ is not divisible by $p$ and since $P$ is cyclic, we have that $r\mid p-1=2$.  
    
    As in \cite[p. 562]{Bra68}, we say that a group with cyclic Sylow subgroups is of metacyclic type if there exists $K\lhd G$ of index $p^ar$, where $|P|=p^a$.
    
    Now, if $G$ is of metacyclic type, then $G$ is $p$-solvable by \cite[Thm.~2C]{Bra76}, and we are in case (i). If $r=1$, we have that $G$ has a normal $p$-complement, as wanted. If $r=2$, the result follows by \cite[Prop.~2F]{Bra76}.

    On the other hand, if $G$ is not of metacyclic type, then  it is not $p$-solvable (again by \cite[Thm.~2C]{Bra76}), and we are in case (ii). The result now follows by \cite[Thm.~3C]{Bra76}.
\end{proof}

\section{Theorem \ref{thm:B}}\label{sec:thm B}

The purpose of this section is to prove Theorem \ref{thm:B}. We start with the following refinement of \cite[Lem.~9.3]{Nav18}.

\begin{lem}\label{lema9.3withtau} Let $\chi\in{\rm Irr}_{p',\tau}(G)$, where $\tau\in\Gal(\QQ^{\mathrm{ab}}/\QQ)$ has $p$-power order, and let $N$ be a normal subgroup of $G$. Then $\chi_N$ has a $P\times\langle\tau\rangle$ invariant constituent, and any two of them are $\norm G P$-conjugate.    
\end{lem}
\begin{proof}
    Let $\mu\in{\rm Irr}(N)$ lying under $\chi$. Since $|G:G_\mu|$ is not divisible by $p$ (by the Clifford correspondence and using that $p$ does not divide $\chi(1)$) and $\langle \tau\rangle$ is a $p$-group, we have that there is at least one $\langle\tau\rangle$-invariant $\theta\in{\rm Irr}(N)$ lying under $\chi$. But then for $g\in G$, we have $(\theta^g)^\tau=(\theta^\tau)^g=\theta^g$ and we have that all of the irreducible constituents of $\chi_N$ are $\langle\tau\rangle$-invariant. Now, by \cite[Lem.~9.3]{Nav18}, we conclude that $\chi_N$ has a $P\times\langle\tau\rangle$-invariant constituent and any two of them are $\norm G P$-conjugate. 
\end{proof}

We will also need the following result of Murai. For the convenience of the reader, we restate the theorem in our notation. Recall that if $B$ is a Brauer $p$-block of $G$ and $\lambda$ is a linear character of $G$, then the set $\lambda B=\{\lambda\chi\>|\> \chi\in{\rm Irr}(B)\}$ is a Brauer $p$-block of $G$ (see \cite[Lem.~2.1]{Riz18}, for instance).

\begin{lem}\label{lem:corollari2.5murai} Let $N$ be a normal subgroup of $G$, let $B$ be a $p$-block of $G$, let $\xi\in{\rm Irr}(G)$ be a linear character and let $\overline{B}$ be a $p$-block of $G/N$. Let $b$ and $\overline{b}$ be the Brauer correspondents of $B$ and $\overline{B}$, respectively. Then $\xi\overline{B}$ is dominated by $B$ if, and only if, $\xi_{N_G(P)}\overline{b}$ is dominated by $b$.
\end{lem}
\begin{proof}
   This is, essentially, \cite[Cor.~2.5]{Mur98}. Write $H=\norm G P$. To apply this result, we just need to check that $\xi^{-1}B$ and $\xi_H^{-1}b$ are Brauer correspondents. But this easily follows from the definition of induced blocks (see \cite[p.~87]{Nav98}). Let $K$ be a conjugacy class of $G$; we claim that $\lambda_{\xi_H^{-1}b}^G(\hat{K})=\lambda_{\xi^{-1} B}(\hat{K}).$  Let $L_1,\ldots, L_r$ be the conjugacy classes of $H$ such that $K\cap H=L_1\cup\ldots\cup L_r$, let $x_i\in L_i$, let $\psi\in{\rm Irr}(b)$ and let $\chi\in{\rm Irr}(B)$. Then
   
   \begin{align*}
       \lambda_{\xi_H^{-1}b}^G(\hat{K})&=\lambda_{\xi_H^{-1}b}\left(\sum_{x\in K\cap H}x\right)=\lambda_{\xi_H^{-1}b}(\hat{L_1}+\ldots +\hat{L_r})\\
       &=\omega_{\xi_H^{-1}b}(\hat{L_1})^*+\ldots + \omega_{\xi_H^{-1}b}(\hat{L_r})^*=\left(\sum_{i=1}^r\frac{|L_i|\overline{\xi_H(x_i)}\psi(x_i)}{\psi(1)}\right)^*
   \end{align*}
Since $\xi$ is a class funcion of $G$, we have that $\xi(x_i)=\xi(x_j)$ for all $i,j$ and then
\begin{align*}
       \lambda_{\xi_H^{-1}b}^G(\hat{K})&=\overline{\xi(x_1)}^*\left(\sum_{i=1}^r\frac{|L_i|\psi(x_i)}{\psi(1)}\right)^*=\overline{\xi(x_1)}^*\lambda_b(\hat{L_1}+\ldots + \hat{L_r})\\
&=\overline{\xi(x_1)}^*\lambda_b(\hat{K\cap H})=\overline{\xi(x_1)}^*\lambda_B(K)=\left(\frac{|K|\xi^{-1}(x_1)\chi(x_1)}{\chi(1)}\right)^*=\lambda_{\xi^{-1}B}(\hat{K}),
   \end{align*}
   proving the claim. 
\end{proof}

The following is Theorem \ref{thm:B}. Recall that $\mathcal{H}$ is the Galois group considered in \cite{Nav04}. 

\begin{thm}\label{thm:ppalblock gam}
    Let $N\normal G$ with $N$ an abelian $p$-group, and assume $G/N$ has cyclic Sylow $p$-subgroups. Let $\tau\in\mathcal{H}$ have $p$-power order. Then $|\irrtau(B_0(G))|=|\irrtau(B_0(\norm G P))|,$ where $P\in\Syl_p(G)$.
\end{thm}
\begin{proof}
Let $\Delta$ be a $\norm G P$-transversal on the set of $P\times \langle\tau\rangle$-invariant elements in $\Irr(N)$. By Lemma \ref{lema9.3withtau}, it follows that
\[\irrtau(B_0(G))=\coprod_{\theta\in\Delta}\irrtau(B_0(G)|\theta),\] and that

\[\irrtau(B_0(\norm G P))=\coprod_{\theta\in\Delta}\irrtau(B_0(\norm G P)|\theta).\]
Hence, it is enough to show that $|\irrtau(B_0(G)|\theta)|=|\irrtau(B_0(\norm G P)|\theta)|$ for every $\theta\in \Delta$. Now, let $\theta\in\Delta$. We claim that the map
$$\iota_G^\theta:\irrtau(B_0(G_\theta)|\theta)\rightarrow\irrtau(B_0(G)|\theta)$$
defined by $\psi\mapsto\psi^G$
 is a bijection. Indeed, since $\theta$ is $\langle \tau \rangle$-invariant, it follows that $\langle \tau\rangle$ acts on $\Irr(G_\theta|\theta)$ and on $\Irr(G|\theta)$, and the Clifford correspondence implies that $\psi\mapsto\psi^G$ is a bijection $\Irr(G_\theta|\theta)\rightarrow\Irr(G|\theta)$. Since the action of $\langle \tau \rangle$ commutes with character induction, it follows that $\tau$ fixes $\psi\in\Irr(G_\theta|\theta)$ if, and only if it fixes $\psi^G\in\Irr(G|\theta)$. Now \cite[Cor.~6.2 and Thm.~6.7]{Nav98} imply the claim. Analogously, since $\norm{G_\theta}{P}=\norm G P_\theta$, the map
$$
\iota_{\norm G P}^\theta:\irrtau(B_0(\norm{G_\theta}{P})|\theta)\rightarrow\irrtau(B_0(\norm G P)|\theta)
$$ defined by $\eta\mapsto\eta^{\norm G P}$
is also a bijection.

Now, we claim that $\theta$ extends to a $\langle \tau\rangle$-invariant character in $\Irr(P)$ if, and only if $\irrtau(B_0(G)|\theta)$ is nonempty. Indeed, let $\chi\in\irrtau(B_0(G)|\theta)$ and let $\psi\in\irrtau(B_0(G_\theta)|\theta)$ be its Clifford correspondent. Then $\psi_P$ is a $\langle \tau \rangle$-invariant character of $P$, and by \cite[Lem.~2.1(ii)]{Nav-Tie19} we have that $\psi_P$ contains a $\langle \tau \rangle$-invariant linear constituent $\lambda\in\Irr(P)$. Since $\psi_N=\psi(1)\theta$, $\lambda$ is an extension of $\theta$. Conversely, if $\theta$ extends to a $\langle \tau \rangle$-invariant character in $\Irr(P)$ then by Lemma \ref{lem:linear sigma-invariant}, $\theta$ extends to a character $\psi\in\irrtau(B_0(G_\theta))$ and $\psi^G\in\irrtau(B_0(G)|\theta)$. In any case, we conclude that $\irrtau(B_0(G)|\theta)\neq\emptyset$ if, and only if  $\theta$ extends to a character $\tilde\theta\in\irrtau(B_0(G_\theta))$ if, and only if, $\theta$ extends to a $\langle \tau\rangle$-invariant character in $\Irr(P)$. Analogously, $\irrtau(B_0(\norm G P)|\theta)\neq\emptyset$ if, and only if  $\theta$ extends to a character $\tilde\theta\in\irrtau(B_0(\norm{G_\theta}P))$ if, and only if, $\theta$ extends to a $\langle \tau\rangle$-invariant character in $\Irr(P)$. In particular, $\irrtau(B_0(\norm G P)|\theta)\neq\emptyset$ if, and only if, $\irrtau(B_0(G)|\theta)\neq\emptyset$. Since we want to prove that $|\irrtau(B_0(G)|\theta)|=|\irrtau(B_0(\norm G P)|\theta)|$, we may assume that both are not empty and therefore, we can find $\tilde\theta\in{\rm Irr}_{0,\tau}(B_0(G_\theta))$   extending $\theta$. Notice that $\tilde\theta':=\tilde\theta_{\norm{G_\theta}P}$ is irreducible and extends $\theta$. Since $|G_\theta:\norm{G_\theta}P|$ is not divisible by $p$, we conclude by \cite[Lem.~2.1]{Lyo-Mar-Nav-Tie25} that $\tilde\theta'\in{\rm Irr}_{0,\tau}(B_0(\norm{G_\theta}P))$ extends $\theta$.

  By Gallagher's theorem, \[\Irr(G_\theta|\theta)=\{\mu\tilde\theta\mid\mu\in\Irr(G_\theta/N)\} \text{ and } \Irr(\norm{G_\theta}P|\theta)=\{\beta\tilde\theta'\mid\beta\in\Irr(\norm{G_\theta}P/N)\}.\] Notice that $\mu\tilde\theta$ is $\langle \tau\rangle$-invariant and has degree coprime to $p$ if and only if $\mu$ is $\langle\tau\rangle$-invariant and has degree coprime to $p$.
Further, \cite[Lem.~2.4]{Riz18} implies that there is a set $\mathcal{B}$ of blocks of $G_\theta/N$ of maximal defect such that 
\[\irrtau(B_0(G_\theta)|\theta)=\coprod_{B\in\mathcal{B}}\{\mu\tilde\theta\mid\mu\in\irrtau(B)\}.\]
If $\mathcal{B}'$ denotes the set of Brauer correspondent blocks in $\norm{G_\theta}P/N=\norm{G_\theta/N}{P/N}$ of the blocks in $\mathcal{B}$, then by Lemma \ref{lem:corollari2.5murai}, we have that
\[\irrtau(B_0(\norm{G_\theta}P)|\theta)=\coprod_{b\in\mathcal{B'}}\{\beta\tilde\theta'\mid\beta\in\irrtau(b)\}.\]

Since $G/N$ has cyclic Sylow $p$-subgroups, \cite[Thm.~3.4]{Nav04} implies that for each $B\in\mathcal{B}$, there is a bijection
$\Omega_B:\irrtau(B)\rightarrow\irrtau(b)$, where $b\in\mathcal{B'}$ is the Brauer correspondent block of $B$. 
It follows that
\begin{align*}
\Omega_\theta:\irrtau(B_0(G_\theta)|\theta)&\rightarrow\irrtau(B_0(\norm{G_\theta}P|\theta)\\
\mu\tilde\theta&\mapsto\Omega_{\bl(\mu)}(\mu)\tilde\theta'
\end{align*}
is a bijection, where $\bl(\mu)$ denotes the block of $G_\theta/N$ that contains $\mu$ (which belongs to $\mathcal{B}$).

Therefore, the map 
\[\Psi_\theta:=\iota_{\norm G P}^\theta\circ\Omega_\theta\circ(\iota_G^\theta)^{-1}:\irrtau(B_0(G)|\theta)\rightarrow\irrtau(\norm G P|\theta)\]
is a bijection, and we can construct a bijection
$\Psi:\irrtau(B_0(G))\rightarrow\irrtau(B_0(\norm G P))$
by setting $\Psi(\chi)=\Psi_\theta(\chi)$ if $\chi$ lies over $\theta\in\Irr_P(N)$. (Notice that the map is surjective because $\irrtau(B_0(G)|\theta)\neq\emptyset$ if, and only if $\irrtau(B_0(\norm G P|\theta)\neq\emptyset$, as proved above.) The result follows.
\end{proof}

\begin{rem}\label{rem:thm B}
    The hypothesis on the Sylow subgroups of $G/N$ being cyclic is only used in the second-to-last paragraph of the proof of Theorem \ref{thm:ppalblock gam}, to apply \cite[Thm. 3.4]{Nav04}. In general, we could assume only that, for every block $B$ of maximal defect of $G/N$ with Brauer correspondent $b$, there is a bijection
    $\Omega_B:\irrtau(B)\rightarrow\irrtau(b)$. Moreover, if one assumes $\Omega_B(\psi)(1)\leq \psi(1)$ for all $\psi\in\irrtau(B)$ and all blocks $B$ of maximal defect, then the proof of Theorem \ref{thm:ppalblock gam} gives a bijection $\Omega:\Irr_{0,\tau}(B_0(G))\rightarrow\Irr_{0,\tau}(B_0(\norm G P))$ with $\Omega(\chi)(1)\leq\chi(1)$. The existence of such bijections has been proposed recently in \cite{Gia25}.
\end{rem}

A similar proof gives the following result.

\begin{pro}\label{thm:ppalblock gam2}
    Let $N\normal G$ be perfect and let $P\in\Syl_p(G)$. Assume every $P$-invariant character $\theta\in\irrsigma(B_0(N))$ that extends to a $\sigma$-invariant character of $PN$ also extends to a $\sigma$-invariant character of $B_0(G_\theta)$ and assume $G/N$ has cyclic Sylow $p$-subgroups. Then $|\irrsigma(B_0(G))|=|\irrsigma(B_0(N\norm G P))|$.
\end{pro}
\begin{proof}
We may argue as in the proof of Theorem \ref{thm:ppalblock gam} that every $\chi\in\irrsigma(B_0(G))$ lies over a $P$-invariant $\theta\in\Irr(B_0(N))$. Since $\theta$ has a number of $G$-conjugates not divisible by $p$ and $\langle \sigma \rangle$ acts on the set of $G$-conjugates of $\theta$, it follows that $\theta$ is also $\sigma$-invariant. Arguing again as in Theorem \ref{thm:ppalblock gam} we conclude that $\chi\in\Irr_{0}(B_0(G))$ is $\sigma$-invariant if and only if its Clifford correspondent $\psi\in\Irr_{0}(B_0(G_\theta))$ is $\sigma$-invariant.

Now each $\theta\in\irrsigma(B_0(N))$ extends to some $\hat\theta\in\irrsigma(B_0(PN))$ by \cite[Cors.~6.2,~6.4]{Nav18} (where we are using that $N$ is perfect so $o(\theta)=1$). By hypothesis, $\theta$ extends to a character in $\irrsigma(B_0(G_\theta))$ and we may follow the proof of Theorem \ref{thm:ppalblock gam} starting at the fourth paragraph to construct a bijection $\irrsigma(B_0(G))\rightarrow\irrsigma(B_0(N\norm G P))$.
\end{proof}

As before, we could obtain more general versions of Proposition \ref{thm:ppalblock gam2}, but given its technicality we have chosen to state it in the simplest way possible. 

\section{Reduction of Theorem \ref{thm:A}}\label{sec:reduction}

In this section, we prove Theorem \ref{thm:A} assuming the results in Section \ref{sec:simples} below. 

\begin{thm}[Navarro--Rizo--Schaeffer Fry--Vallejo]\label{thm:3.4 of NRSV}
Suppose that $G=NP$, where $N$ is a nonabelian nonsimple
minimal normal subgroup of $G$, $\cent GN=1$, $G/N$ is cyclic
and $P \in \syl 3G$.  Let $S\normal N$ be simple and suppose that the Sylow $3$-subgroups of $S$ are not cyclic, $H=\norm GS$, $C=\cent GS$
and let $V \in \syl 3{H/C}$. Then the following hold.
\begin{enumerate}[{\rm (a)}]
\item
$|{\rm Irr}_{3',\sigma}(B_0(G))|\in\{6,9\}$ if and only if  $H>SC$ and $|{\rm Irr}_{3',\sigma}(B_0(H/C))|\in\{6,9\}.$
\item
$|P:\Phi(P)|=9$ if and only if $H>SC$ and $|V:\Phi(V)|=9$.
\end{enumerate}
\end{thm}
\begin{proof}
Mimic the proof of \cite[Thm.~3.4]{Nav-Riz-Sch-Val20}. Notice that \cite[Thm.~3.4]{Nav-Riz-Sch-Val20} did not require the hypothesis on non-cyclic Sylow subgroups of $S$, because nonabelian simple groups never have cyclic Sylow $2$-subgroups.
\end{proof}

Next is the proof of Theorem \ref{thm:A}, assuming the results of Section \ref{sec:simples}. For the remainder of the section, we let $p=3$. 

\begin{thm}
    Let $G$ be a finite group, let $B_0(G)$ be the principal 3-block of $G$, and let $P\in\syl 3 G$. Suppose that $k_{0,\sigma}(B_0(G))\in\{6,9\}$. Then $|P:\Phi(P)|=9.$
\end{thm}
\begin{proof}
    We argue by induction on $|G|$. Since $\irrsigma(B_0(G))=\irrsigma(B_0(G/\oh{p'}G))$ by \cite[Thm.~9.9]{Nav98}, we may assume $\oh{p'}G=1$. We can also assume that $P$ is not cyclic by \cite[Thm.~A]{Riz-Sch-Val20}. 

    \medskip

    \textit{Step 1: We may assume $P$ is not normal in $G$}.
    
    If $P\normal G$, since $\oh{p'}G=1$, we have that $|\irr{G/\Phi(P)}|\in\{6,9\}$ (see \cite[Lem.~2.2]{Riz-Sch-Val20}, for instance). Now we conclude by examining the structure of the Sylow 3-subgroups of groups having 6 or 9 conjugacy classes and normal Sylow $3$-subgroups (see \cite{Ver-Ver85}) (notice that $|P/\Phi(P)|>p$, since otherwise $P$ is cyclic).

\medskip
Let $N$ be a minimal normal subgroup of $G$.
\medskip

    \textit{Step 2: We may assume $N$ is not $p$-elementary abelian.}

    Assume $N$ is $p$-elementary abelian. We have the following cases.
    \begin{enumerate}
        \item $\ksigma(B_0(G/N))=1$. In this case $|G/N|$ is not divisible by $p$ by Lemma \ref{lem:p-action}, so $P=N\normal G$ and we are done by Step 1.
          \item $\ksigma(B_0(G/N))=\ksigma(B_0(G))$. In this case $N\sbs\ker\chi$ for all $\chi\in\irrsigma(B_0(G))$. By Lemma \ref{lem:kernel of sigmas}, $N\sbs\Phi(P)$. By induction, $|P:\Phi(P)|=|P/N:\Phi(P)/N|=9$ and we are done.
        \item $\ksigma(B_0(G/N))=6$. By the previous case, $\ksigma(B_0(G))=9$. In this case, by induction we have $|P/N:\Phi(P)N/N|=9$ and it suffices to show that $N\subseteq \Phi(P)$. Assume otherwise, and let $\lambda\in{\rm Irr}(P/\Phi(P))$, and suppose that $1_N\neq\theta=\lambda_N\in{\rm Irr}(N)$ (here we are using that $N\not\sbs\Phi(P)$). Now, $\theta$ extends to $\lambda\in{\rm Irr}_{0,\sigma}(P)$ and we can apply Lemma \ref{lem:linear sigma-invariant}, then $\theta$ extends to $\hat\theta\in{\rm Irr}_{0,\sigma}(B_0(G_\theta))$. By Proposition \ref{pro:noelia relative sigma} and the hypothesis, we have that $\ksigma(B_0(G_\theta)|\theta)=3$. Since $\hat\theta$ is $\sigma$-invariant, by Gallagher's theorem and \cite[Lem.~2.4]{Riz18}, $\ksigma(B_0(G/N))\leq \ksigma(B_0(G_\theta)|\theta)=3<6=\ksigma(B_0(G/N))$, a contradiction.
        \item $\ksigma(B_0(G/N))=3$. In this case, $G/N$ has cyclic Sylow $3$-subgroups by the main result of \cite{Riz-Sch-Val20}, so by Theorem \ref{thm:ppalblock gam} we have $\ksigma(B_0(G))= \ksigma(B_0(\norm G P))$ and we are done by Step 1.
    \end{enumerate}

    \medskip

Therefore we may assume $N=S_1\times\dots\times S_t$ where the $S_i$'s are nonabelian simple groups of order divisible by $p$.

\medskip

    \textit{Step 3: If $\ksigma(B_0(G/N))=1$ then we are done.}
    
    In this case $G/N$ is not divisible by $p$ and hence $P\subseteq N$. Since $\oh{p'}G=1$, we also have that $\cent G N=1$. If $t=1$, we conclude by applying the main result of \cite{eden}. Hence, we may assume that $t>1$. By the Frattini argument, we have that $G=N\norm G P$, so $M=N\cent G P$ is normal in $G$. Then ${\rm Irr}(G/M)\subseteq{\rm Irr}_{0,\sigma}(B_0(G))$ (see \cite[Lem.~2.1]{Gia-Riz-Sch-Val24}). Then, by Theorem \ref{thm:simple group conditions}, there are at least 5 character degrees in ${\rm Irr}_{0,\sigma}(B_0(N))$, and so the same holds in ${\rm Irr}_{0,\sigma}(B_0(M))$ by Theorem \ref{thm:Alperin-Dade}. For every $\psi\in{\rm Irr}_{0,\sigma}(B_0(M))$ there is at least one $\chi\in{\rm Irr}_{p',\sigma}(B_0(G))$ over it, so the set of character degrees of ${\rm Irr}_{0,\sigma}(B_0(M))\setminus\{1_M\}$ has size at most $9-k(G/M)$. This, together with the fact that $|G/M|$ is not divisible by 3, forces $k(G/M)\in\{1,2,4,5\}$. Now, we may argue as in \cite[Thm.~6.3]{Gia-Riz-Sch-Val24}, Step 2, cases (2.a.i)--(2.a.iv) using Theorem \ref{thm:almost simple}(iii) instead of Theorem 3.1(c) in loc. cit. to finish:

    \begin{enumerate}
        \item Case $k(G/M)=1$. In this case $G=M$, and by Theorem \ref{thm:Alperin-Dade}, we obtain $$\ksigma(B_0(S))^t=\ksigma(B_0(N))=\ksigma(B_0(M))=9.$$ This forces $t=2$ and $\ksigma(B_0(S))=3$. By \cite[Thm.~A]{Riz-Sch-Val20}, this means that $S$ has cyclic Sylow subgroups and hence $P\sbs N\cong S\times S$ is 2-generated, as wanted.
        \item Case $k(G/M)=2$. In this case, arguing as in (2.a.ii) of \cite[Thm.~6.3]{Gia-Riz-Sch-Val24}, we obtain again that $\ksigma(B_0(S))=3$ and $t=2$. We remark that one needs to apply \cite[Lem.~4.2]{Nav-Tie21} to show that there is a height-zero $\sigma$-invariant $\chi\in\Irr(B_0(G))$ lying above every character in $\irrsigma(B_0(M))$.
        \item Case $k(G/M)\in\{4,5\}$. Arguing as in cases (2.a.iii) and (2.a.iv) of \cite[Thm.~6.3]{Gia-Riz-Sch-Val24} (and applying again \cite[Lem.~4.2]{Nav-Tie21}), we obtain that either $t=2$ and $\ksigma(B_0(S))=3$ (and we conclude as above), or $t=2$ and $\ksigma(B_0(S))=6$. In this case, $S$ is not cyclic (by the main result of \cite{Riz-Sch-Val20}, and hence $9\mid |S|$. We conclude as in cases (2.a.iii) and (2.a.iv) by using Theorem \ref{thm:almost simple}(iii).
    \end{enumerate}
    
\medskip

    \textit{Step 4: If $\ksigma(B_0(G/N))=\ksigma(B_0(G))$ then we are done.}
    
    In this case, every character in $\irrsigma(B_0(G))$ contains $N$ in its kernel. Since $P$ acts on $\irrsigma(B_0(N))$ (a set of size divisible by $3$ by Lemma \ref{lem:p-action}) and fixes at least one character, then it fixes at least $3$. Let $\eta\in\irrsigma(B_0(N))$ be nontrivial and $P$-invariant. By \cite[Cors.~6.2,~6.4]{Nav18} there is a $\psi\in\irrsigma(PN|\eta)$ extending $\eta$. Then $\psi$ extends to $\hat\psi\in\irrsigma(PN\cent G P|\eta)$ by Theorem \ref{thm:Alperin-Dade}, and $\hat\psi^G$ contains some $\chi\in\irrsigma(B_0(G))$ by \cite[Lem.~2.4]{Riz-Sch-Val20}. Since $\chi_N$ contains $\eta$ then $\chi\in\irrsigma(B_0(G))-\irrsigma(B_0(G/N))$, a contradiction.

\medskip

    \textit{Step 5: If $\ksigma(B_0(G/N))=6$ then we are done.}
    
In this case, we may assume $\ksigma(B_0(G))=9$ by Step 4. We first prove that $P$ acts transitively on the simple direct factors of $N$. Suppose the contrary and let $N=R\times T$, where $R$ is the direct product of the elements in a $P$-orbit of $\{S_1,\ldots, S_t\}$. Since $p$ divides $|R|$ and $|T|$, we know that $p$ divides $\ksigma(B_0(R))$ and $\ksigma(B_0(T))$, so there are at least two nontrivial $P$-invariant characters $\theta_R,\xi_R\in\irrsigma(B_0(R))$ and two nontrivial $P$-invariant characters $\theta_T,\xi_T\in\irrsigma(B_0(T))$. Now, $1_N,1_R\times \theta_T$ and $\theta_R\times\theta_T$ are $P\times\langle\sigma\rangle$-invariant and not $G$-conjugate. But this is a contradiction, since there are at least 6 irreducible characters in $\irrsigma(B_0(G))$ lying over $1_N$ and there are at least 3 distinct characters in $\irrsigma(B_0(G))$ lying over each $1_R\times \theta_T$ and $\theta_R\times\theta_T$ by Proposition \ref{pro:noelia relative sigma}. Hence $P$ acts transitively on $\{S_1,\ldots, S_t\}.$

 Now, by induction we have that $|PN:\Phi(P)N|=9$. Suppose that Theorem \ref{thm:simple group conditions}(i)(a) holds for $S$ and let $\theta_1,\theta_2$ be as therein. By Lemma \ref{lem:princblockabove}(v), let \[\tilde{\theta_i}=\theta_i\times\dots\times\theta_i\in{\rm Irr}_{p',\sigma}(B_0(N)),\] so $\tilde{\theta}_1$ and $\tilde{\theta}_2$ are $P$-invariant. Note that $\tilde\theta_1$ and $\tilde\theta_2$ extend to characters in $\irrsigma(B_0(PN))$ by \cite[Cors.~6.2,~6.4]{Nav18}, and that $\tilde\theta_1$ and $\tilde\theta_2$ are not $G$-conjugate. By Proposition \ref{pro:noelia relative sigma} each of these characters has at least 3 irreducible characters in $\irrsigma(B_0(G))$ lying above them, a contradiction.
Finally suppose that Theorem \ref{thm:simple group conditions}(i)(b) holds for $S$ and let $\theta\in{\rm Irr}_{p'}(B_0(S))$ be the character given by the statement of Theorem \ref{thm:simple group conditions}(i)(b).  
Let $\tilde{\theta}\in{\rm Irr}_{p'}(B_0(N))$ be the product of copies of $\theta$, so that $\tilde{\theta}$ is $P$-invariant. Notice then that $G_{\tilde \theta}$ acts transitively on the simple direct factors of $N$ (because $P$ does). By Lemma \ref{lem:tensor induction sigma} we also know that $\tilde \theta$ extends to some $\phi \in\irrsigma(B_0(G_{\tilde \theta}))$. Since $\phi$ is $\sigma$-invariant, by Gallagher's theorem and \cite[Lem.~2.4]{Riz18}, we have that $\ksigma(B_0(G_{\tilde \theta})|\tilde{\theta})\geq \ksigma(B_0(G_{\tilde \theta}/N))$. By Clifford's theorem, $\ksigma(B_0(G)|\tilde\theta)\geq\ksigma(B_0(G_{\tilde \theta}/N))$. Since $\ksigma(B_0(G/N))=6$ we have that $PN/N$ is not cyclic, then  $\ksigma(B_0(G_{\tilde{\theta}}/N))>3$ by applying \cite[Thm.~A]{Riz-Sch-Val20} twice, and this gives a contradiction as $\ksigma(B_0(G))-\ksigma(B_0(G/N))= 3$.

\medskip

Therefore we have that $\ksigma(B_0(G/N))=3$ and in particular $G/N$ has cyclic Sylow $3$-subgroups by \cite[Thm.~A]{Riz-Sch-Val20}. Recall that $N=S_1\times\dots\times S_t$, where $S_i\cong S$, nonabelian simple group of order divisible by $p$.

\medskip

\textit{Step 6: $N$ is the unique minimal normal subgroup of $G$ and $t\geq 2$.}

 If $M$ is another minimal normal subgroup, then by Steps 3, 4 and 5 we may assume that $\ksigma(B_0(G/M))=3.$ By \cite[Thm.~A]{Riz-Sch-Val20} we have that $N\cong NM/M$ and $G/N$ have cyclic Sylow 3-subgroups, then $P$ is metacyclic but not cyclic, and thus it is $2$-generated, as wanted.  If $N$ is simple then $G$ is almost simple and we are done by Theorem \ref{thm:almost simple}, so we assume $t\geq 2$.

\medskip

Our next goal is to show that $PN$ acts transitively on the set $\{S_1,\dots,S_t\}$.  We write $N=M_1\times\dots\times M_s$ where each $M_i$ is the direct product of a single $P$-orbit on $\{S_1,\dots,S_t\}$.

\medskip

      \textit{Step 7: We may assume $s\leq 2$.}

  Since $M_i\normal PN$, for each $M_i$ we can find a nontrivial $P$-invariant $\theta_i\in\irrsigma(B_0(M_i))$, using Lemma \ref{lem:p-action}. Then, by Lemma \ref{lem:princblockabove}(v), we can produce $s+1$ characters
    $$1_N, (\theta_1\times 1_{M_2}\times\dots\times 1_{M_s}), (\theta_1\times\theta_2\times 1_{M_3}\times\dots\times 1_{M_s}),\dots,(\theta_1\times\dots\times \theta_s)$$
    that lie in $\irrsigma(B_0(N))$ and are $P$-invariant and not $G$-conjugate. Since $N$ is perfect, all of these extend to characters in $\irrsigma(B_0(PN))$ by \cite[Cors.~6.2,~6.4]{Nav18} and Lemma \ref{lem:princblockabove}(iv). Proposition \ref{pro:noelia relative sigma} produces $3$ distinct characters in $\irrsigma(B_0(G))$ above each of these, so $\ksigma(B_0(G))\geq 3(s+1)$. If $s>2$ we arrive at a contradiction.

\medskip

    \textit{Step 8: We may assume $s\neq 2$.}
    
    If $s=2$ then $N=M_1\times M_2$ and $M_1$ and $M_2$ are $G$-conjugate but not $P$-conjugate. Suppose that Theorem \ref{thm:simple group conditions}(i)(a) holds for $S_i$ and let $\theta_1$ and $\theta_2$ be the characters as therein. Let $\tilde\theta_i=\theta_i\times\cdots\times\theta_i\in{\rm Irr}(B_0(M_1))$, using Lemma \ref{lem:princblockabove}(v).  Then  the characters
    $$1_N, \tilde\theta_1\times 1_{M_2},\tilde\theta_2\times 1_{M_2}$$
    are not $G$-conjugate. Further, they are not conjugate to the character $\eta\in\irrsigma(B_0(N))$ constructed as the direct product of copies of $\theta_1$. Now all the characters $1_N, \tilde\theta_1\times 1_{M_2},\tilde\theta_1\times 1_{M_2}$ and $\eta$ are $P$-invariant so by \cite[Cors.~6.2,~6.4]{Nav18} and Lemma \ref{lem:princblockabove}(iv), they extend to characters in $\irrsigma(B_0(PN))$. By Proposition \ref{pro:noelia relative sigma}, we have that $\ksigma(B_0(G))\geq 12$, which is absurd.

    Therefore we assume that the $S_i$ are as in Theorem \ref{thm:simple group conditions}(i)(b). Write $S=S_1$ and let $\theta\in{\rm Irr}(S)$ be the character of $S$ given in that statement. Let $\eta_i=\theta\times\theta\times\cdots\times\theta\in\irrsigma(B_0(M_i))$ (using Lemma \ref{lem:princblockabove}(v)). We claim that every $P$-invariant character in $\irrsigma(B_0(N))$ is $G$-conjugate to $1_N$, $\eta_1\times 1_{M_2}$ and $\eta_1\times \eta_2$. Indeed, if there is $\xi\in\irrsigma(B_0(N))$ not $G$-conjugate to any of these, then $\xi$ extends to $\irrsigma(B_0(PN))$ by \cite[Cors.~6.2,~6.4]{Nav18} and Lemma \ref{lem:princblockabove}(iv). By Proposition \ref{pro:noelia relative sigma} we have that $\ksigma(B_0(G))\geq 12$, a contradiction that proves the claim. Moreover, notice that $\irrsigma(B_0(G)|1_N)=3$.

    Let $L=\norm G{M_1}=\norm G {M_2}$ and notice that $|G:L|=2$,  $G_{\eta_1\times 1_{M_2}}=L_{\eta_1\times 1_{M_2}}=L_{\eta_1}$ and $M_2\sbs L_{\eta_1}$. Since $PN$ acts transitively on the copies of $S$ inside $M_1$ and $PN\sbs L_{\eta_1}$, by Lemma \ref{lem:tensor induction sigma} we have that $\eta_1$ extends to $\hat\eta_1\in\irrsigma(B_0(L_{\eta_1}))$. Now since $M_2$ is perfect $(\hat\eta_1)_N=\eta_1\times 1_{M_2}$. Similarly, $1_{M_1}\times \eta_2$ extends to $\hat\eta_2\in\irrsigma(B_0(L_{\eta_2}))$.

    We claim that $\eta=\eta_1\times\eta_2$ extends to a character in $\irrsigma(B_0(G_{\eta}))$. If $G_{\eta}L=G$ then $G_{\eta}$ acts transitively on the set $\{S_1,\dots, S_t\}$ and we are done applying Lemma \ref{lem:tensor induction sigma}. Otherwise, $G_{\eta}\sbs L$ and therefore $G_{\eta}= L_{\eta_1}\cap L_{\eta_2}$. Now, since $PN\leq G_{\eta}$, $G_{\eta}$ also acts transitively on the copies of $S$ contained in $M_1$ and in $M_2$. Repeating the previous argument, $\eta_1\times 1_{M_2}$ and $1_{M_1}\times \eta_2$ extend to $\mu_1,\mu_2\in\irrsigma(B_0(G_{\eta}))$ respectively. Now, by \cite[Lem.~2.3]{Gia-Riz-Sam-Sch20}, $\mu_1\mu_2\in\irrsigma(B_0(G_{\eta}))$ and
    $$(\mu_1\mu_2)_{N}=(\eta_1\times 1_{M_2})(1_{M_1}\times \eta_2)=\eta$$
    and we are done.

    Therefore, we are in the case of the hypotheses of Proposition \ref{thm:ppalblock gam2}, so \[\ksigma(B_0(G))=\ksigma(B_0(N\norm G P))\] so if $N\norm G P\leq G$ we are done by induction. Otherwise $PN\normal G$.

    Let $X=PN\cent G P$. Then $\Irr(G/X)\sbs\irrsigma(B_0(G)|1_N)$. By the previous paragraphs and Proposition \ref{pro:noelia relative sigma}, we have $\ksigma(B_0(G)|\eta_1\times 1_{M_2})=3=\ksigma(B_0(G)|\eta_1\times\eta_2)$, and this implies that $\ksigma(B_0(G)|1_N)=3$. Now $p$ does not divide $|G:X|$, so we have $|G:X|\leq 2$. By Theorem \ref{thm:Alperin-Dade}, $\ksigma(B_0(X))=\ksigma(B_0(PN))$. If $X=G$, by induction we may assume $PN=G$ but then $N$ is minimal normal in $PN$ and $P$ acts transitively on $\{S_1,\dots, S_t\}$ so we are done. Otherwise, $|G:X|=2$, so $\irrsigma(B_0(X))\leq 15$. By Theorem \ref{thm:Alperin-Dade} and \cite[Lem.~2.5]{Nav-Riz-Sch-Val20} we have that    $$3|\irrpsigma(B_0(N))|=|\irrsigma(B_0(PN))|=|\irrsigma(X)|\leq 15.$$ Since $3$ divides $|\irrpsigma(B_0(N))|$, this forces $|\irrpsigma(B_0(N))|=3$, so $|\irrsigma(X)|=9$ and we are done by induction.

\medskip

\textit{Final step.}

We conclude that $s=1$ so $PN$ acts transitively on the set $\{S_1,\dots,S_t\}$. Recall that $PN/N$ is cyclic.

Assume first $S_i$ has noncyclic Sylow $p$-subgroups. Consider $M=\bigcap \norm G {S_i}\normal G$ and notice that $M\sbs \norm G S$. Since $p$ divides $|G:\norm G S|$ then $p$ divides $|G:M|$. By Theorem \ref{Bra76}, either $G/N$ is $p$-solvable or $|M/N|$ is not divisible by $p$. 

Suppose first that $|M/N|$ is not divisible by $p$. Notice that, since $PN/N$ is abelian and $PN$ acts transitively on $\{S_1,\ldots, S_t\}$, then $\norm {PN}{S_i}=\norm{PN}{S_j}$ for every $i,j$. Then, $N\sbs\norm {PN} S\sbs M$, and therefore $N=\norm {PN}{S}$. Now, we can apply Lemma \ref{lem:carolina} and Theorem \ref{thm:simple group conditions}(iii) to obtain three $P$-invariant characters in $\irrsigma(B_0(N))$ that are not $G$-conjugate. Now, Proposition \ref{pro:noelia relative sigma} gives 9 characters in $\irrsigma(B_0(G))$ not in $\irrsigma(B_0(G/N))$, a contradiction.

Hence, we may assume that $G/N$ is $p$-solvable. Theorem \ref{Bra76} implies that there exists $L\leq G$ with $|G:L|\in\{1,2\}$ such that $L/N$ has a normal $p$-complement $K/N$. If $|G:L|=2$, then $\ksigma(B_0(L))\leq 16$ and since $\ksigma(B_0(L))$ is divisible by 3, we obtain that  $\ksigma(B_0(L))\in\{3, 6, 9, 12, 15\}$. Notice that $K=\Oh{p}L$, and hence, by \cite[Lem.~2.5]{Nav-Riz-Sch-Val20}, we obtain that $k_{0,\sigma,P}(B_0(K))\in\{1,2,3,4,5\}$. Since 3 divides $|K|$, $k_{0,\sigma,P}(B_0(K))\equiv\ksigma(B_0(K))\equiv 0$ mod 3, and then the only possibility is $k_{0,\sigma,P}(B_0(K))=3$. But in this case $\ksigma(B_0(L))=9$ and we are done by induction.

Hence we may assume that $G=L=KP$, with $K\normal G$. Now by \cite[Thm.~2.8]{Nav-Riz-Sch-Val20}, we have that $\ksigma(B_0(G))=\ksigma(B_0(N\norm G P)$ and therefore, by induction, we can assume that $PN\normal G$. Now consider $X=PN\cent G P$ and notice that $\irr {G/X}\sbs \irrsigma(B_0(G)|1_N)$ by \cite[Lem.~2.1]{Gia-Riz-Sch-Val24}. Since $p$ divides $|N|$, we have that $k_{0,\sigma,P}(B_0(N))\equiv\ksigma(B_0(N))\equiv 0$ mod $p$. By Proposition \ref{pro:noelia relative sigma}, we obtain that $|\irrsigma(B_0(G)|1_N)|=3$. Since $|G:X|$ is not divisible by 3, we obtain that $|G:X|\in\{1,2\}$. If $|G:X|=2$, then $$p\cdot k_{0,\sigma,P}(B_0(N))=\ksigma(B_0(PN))=\ksigma(B_0(X))\leq 15,$$ where we have used \cite[Lem.~2.5]{Nav-Riz-Sch-Val20} in the first equality and Theorem \ref{thm:Alperin-Dade} in the second equality. This forces $|\irrpsigma(B_0(N))|=3$ and then $\ksigma(B_0(PN))=9$. By induction we may assume $PN=G$. If $G=X$, by Theorem \ref{thm:Alperin-Dade} we may assume $PN=G$ aswell, and the result follows from Theorem \ref{thm:3.4 of NRSV} and the main result of \cite{eden}.

  We are left with the case that the $S_i$'s have cyclic Sylow $p$-subgroups, and let $C_1\in\Syl_p(S_1)$. Let $C_i=C_1^{y^i}\in\Syl_p(S_1^{y^{i}})$, $Q=C_1\times\dots\times C_t\in\Syl_3(N)$ where each $C_i\in\Syl_3(S_i)$ and where $Q\sbs P$, with $P/Q=\langle yQ\rangle$ cyclic. Then $\langle yQ\rangle$ permutes $\{C_1,\dots,C_t\}$ transitively, and since $C_i$ is normal in $Q$ for all $i$, we conclude that $\langle y\rangle$ permutes $\{C_1,\dots,C_t\}$ transitively. This implies that there are generators $x_1,\dots, x_t$ of $C_1,\dots, C_t$ respectively such that $\langle y\rangle$ permutes the set $\{x_1,\dots,x_t\}$ transitively. Write $z_i=(1,\dots, 1, x_i,1,\dots, 1)$. Notice that $Q=\langle z_1,\dots, z_t\rangle$. We claim that $P=\langle y, z_1\rangle$. Since $\langle y\rangle$ acts transitively on $\{x_1,\dots,x_t\}$ then it acts transitively on $\{z_1,\dots,z_t\}$. It follows that every $z_j\in\langle y,z_1\rangle$, so $Q\sbs\langle y, z_1\rangle$ and $P=\langle y,Q\rangle\sbs\langle y, z_1\rangle$. Since $P$ is not cyclic, it is $2$-generated and we are done.
\end{proof}

\section{Simple groups}\label{sec:simples}

The goal of this section is to prove the following ``$\sigma$-version" of \cite[Thm.~3.1]{Gia-Riz-Sch-Val24}.

\begin{thm}\label{thm:simple group conditions}
Let $S$ be a nonabelian simple group of order divisible by $3$ and let $X\in\Syl_3(\Aut(S))$. 
\begin{enumerate}
    \item One of the following holds.
    \begin{enumerate}
        \item There exist $1_S\neq\theta_1,\theta_2\in\irrsigma(B_0(S))$ nonconjugate in $\Aut(S)$ and invariant under $X$, or
        \item There is an $X$-invariant $1_S\neq \theta\in\irrsigma(B_0(S))$ that extends to some $\sigma$-invariant character in $B_0(T)$ for all $S\leq T\leq \Aut(S)_\theta$.
    \end{enumerate}
    \item The set of degrees of characters in $\irrsigma(B_0(S))$ has size at least $3$.
    \item If $S$ has noncyclic Sylow $3$-subgroups, then there exist $1_S\neq \theta_1,\theta_2,\theta_3\in\irrsigma(B_0(S))$ nonconjugate in $\Aut(S)$. 
\end{enumerate}
\end{thm}

We remark that part (iii) of Theorem \ref{thm:simple group conditions} is already found as \cite[Thm.~C(b)]{Riz-Sch-Val20}.
In addition to Theorem \ref{thm:simple group conditions}, we will need the main result from \cite{eden}, restated here:

\begin{thm}[E. Ketchum]\label{thm:almost simple}
Let $A$ be an almost simple group,  and let $P\in\Syl_3(A)$. Then $\ksigma(B_0(A))\in\{6,9\}$ if and only if $|P:\Phi(P)|=9$.
\end{thm}

\subsection{Previous Results}

We begin by recording some previous results that will be useful for proving Theorem \ref{thm:simple group conditions}.

\begin{lem}\label{lem:MMSV3.5}
Let $N\lhd G$ with $3\nmid [G:N]$. 
\begin{itemize}
    \item If  $\theta \in \irr N$ and $\chi \in \irr {G\mid N}$, then
$\chi$ is $\sigma$-fixed if and only if $\theta$ is 
$\sigma$-fixed.
\item If $\theta\in\Irr_{\sigma}( {B_0(N)})$, then there is some $\chi\in\Irr_{\sigma}(B_0(G)|\theta)$.
\end{itemize}

\end{lem}
\begin{proof}
This is from \cite[Lem.~3.5, Cor.~3.6]{Mar-Mar-Sch-Val24}.
\end{proof}

Assume $S=G/\zent{G}$ is a simple group of Lie type, with $G=\bG^F$, where $\bG$ is a simple, simply connected algebraic group defined over a field of characteristic $q_0$ and $F\colon \bG\rightarrow\bG$ is a Steinberg endomorphism. Let $(\bG^\ast, F)$ be dual to $(\bG, F)$ and write $G^\ast:=(\bG^\ast)^F$.

\begin{lem}\label{lem:defcharsigfixed}
Let $S$ be as above with $q_0=p\geq 3$.  Then every $\chi\in\irr{S}$ is $\sigma$-invariant.
\end{lem}
\begin{proof}
This follows from \cite[Thm. 1.3 and Prop. 10.12]{TZ04}.
\end{proof}

\begin{lem}\label{lem:unipsfixed}
Assume $S$ is as above with $q_0\neq p$ and $p\geq 3$. Then any unipotent character of $S$ is $\sigma$-invariant.
\end{lem}
\begin{proof}
    This is \cite[Lem.~4.7]{Mar-Mar-Sch-Val24}.
\end{proof}

\begin{lem}\label{lem:semisimplep}
Assume $S$ and $G$ are as above with $q_0\neq p$ and $p$ good for $G$.  Assume $s\in G^\ast$ has order $p$ and $p\nmid |\zent{G}|$. Then the semisimple character $\chi_s$ deflates to a character in $\Irr_{\sigma}(B_0(S))$.
\end{lem}
\begin{proof}
This is exactly as in \cite[Lem.~4.6]{Mar-Mar-Sch-Val24} and \cite[Thm.~5.1]{HSF}, which uses \cite[Cor.~ 3.4]{hiss} and \cite[Thm.~21.13]{CE04}. Note that $\chi_s$ is trivial on $\zent{G}$ since $s\in [G^\ast, G^\ast]$ (see \cite[Lem.~4.4]{Nav-Tiep16}), and $\cent{\bG^\ast}{s}$ is connected, using \cite[Ex.~20.16]{MT11}.
\end{proof}

\subsection{Proof of Theorem \ref{thm:simple group conditions}}

\begin{lem}\label{lem:sporalt}
Theorem \ref{thm:simple group conditions} holds when $S$ is an alternating or sporadic group, a group of Lie type defined in characteristic $3$, or $\tw{2}\type{F}_4(2)'$. Further, (i)(a) holds unless $S\in\{\PSL_2(3^a), \PSL_3^\epsilon(3^a)\}$.
\end{lem}
\begin{proof}
Let $S$ be as in the statement. If $S = J_3$, we see the statement from computation in \cite{GAP}, and we see that (i)(a)  holds in this case. So, we now assume that $S\neq J_3$. 

 Lemma \ref{lem:defcharsigfixed}  and \cite[Lem.~4.1]{Mar-Mar-Sch-Val24} give that Theorem \ref{thm:simple group conditions} is equivalent to \cite[Thm.~3.1]{Gia-Riz-Sch-Val24} in these cases whenever condition (a1) of the latter holds. Thus, by \cite[Rem.~3.2]{Gia-Riz-Sch-Val24}, we need only consider groups of the form $\PSL_2(3^a)$ and $\PSL^{\epsilon}_3(3^a)$. Further, we need only verify that condition (i)(b) of Theorem \ref{thm:simple group conditions} holds in these cases.

Let $S = \PSL_2(3^a)$ or $S = \PSL^{\epsilon}_3(3^a)$. From \cite[Thm.~3.1 and Rem.~3.2]{Gia-Riz-Sch-Val24} we obtain some $1\neq \theta \in \Irr_0(B_0(S))$, which is $X$-invariant, such that for every $S\le T\le \Aut(S)_\theta$ we have that $\theta$ extends to some character in $B_0(T)$ (in both cases $\Aut(S)_\theta/S$ is abelian, so in fact all characters lying above $\theta$ in $\Aut(S)_\theta$ are extensions). Further, we have that $\theta $ is $\sigma$-fixed by Lemma \ref{lem:defcharsigfixed}.  Let $S\le T \le \Aut(S)_\theta$. In both cases we observe that $T/S$ has a normal Sylow $3$-subgroup. Let $M$ be the preimage under the natural projection $\pi:T\rightarrow T/S$ of this Sylow $3$-subgroup. Since $S$ is perfect, we have the determinantal order $o(\theta)$ satisfies $o(\theta)=1$ and we can  use \cite[Cor.~6.4]{Nav18} to obtain a $\sigma$-fixed extension of $\theta$ in the principal block of $M$.  (Recall that since $M/S$ has $3$-power order, $B_0(M)$  is the only block above $B_0(S)$ by Lemma \ref{lem:princblockabove}.) Then Lemma \ref{lem:MMSV3.5}
gives a $\sigma$-fixed extension in the principal block of  $T$. 
\end{proof}

\begin{lem}\label{lem:nondef}
Theorem \ref{thm:simple group conditions} holds when $S$ is a group of Lie type defined in characteristic $p\neq 3$, with (i)(a) holding unless $S=\PSL_2(q)$ or $3| (q +\epsilon)$ and $S=\PSL^{\epsilon}_3(q)$. 
\end{lem}

\begin{proof}
Assume
 $S$ is a simple group of Lie type defined in characteristic $p\neq 3$.

    First we consider (i). Assume $S= \PSL_2(q)$ or that $3|(q+\epsilon)$ and $S= \PSL^{\epsilon}_3(q)$. We then claim that $(b)$ holds. Take $\theta = \hat\theta_S$, with $\hat\theta$ being the Steinberg character in $\tilde{S}$, where $\wt{S}$ is $\PGL_2(q)$ and $\PGL_3^{\epsilon}(q)$ in the respective cases. Then $\theta$ is $\Aut(S)$-invariant by \cite[Thm.~2.5]{Mal08} and extends to $\Aut(S)$ by \cite[Thm.~2.4]{Mal08} and we see that $\hat\theta \in\Irr_\sigma(B_0(\wt{S}))$ by arguing as in the second and third paragraphs of the proof of \cite[Prop.~3.9]{Riz-Sch-Val20}. 
    Let $S\le T \le \Aut(S)_\theta = \Aut(S)$. We clearly have that $\theta$ extends to   $\hat \theta _{T\cap \tilde{S}} \in \Irr(B_0(T\cap \tilde{S}))$. Since the Steinberg character is rational-valued, so is $\hat\theta_{T\cap\wt{S}}$. We have that $T/(T\cap \tilde{S})$ is an abelian group and therefore has a  normal Sylow $3$-subgroup. Let $M$ be the preimage under $\pi:T\rightarrow T/(T\cap \tilde{S})$ of this Sylow $3$-subgroup. Then $B_0(M)$ is the unique block above $B_0(T\cap \wt{S})$  by Lemma \ref{lem:princblockabove} and we use \cite[Cor.~6.6(a)]{Nav18} to obtain a rational extension $\chi$ of $\hat\theta_{T\cap\tilde{S}}$ in $\Irr(B_0(M))$. Finally, using Lemma \ref{lem:MMSV3.5}, we obtain a character $\psi$ in $\Irr_{\sigma}(B_0(T)|\chi)$. Since $\psi$ lies above $\hat\theta_{T\cap \tilde{S}}$ and $T/(T\cap \wt{S})$ is abelian, it follows that $\psi$ must be an extension of $\hat\theta_{T\cap\wt{S}}$ (and hence of $\theta$), as desired. 

Now Assume $S$ is not one of the groups discussed above.  Then  \cite[Prop.~5.8]{Gia-Riz-Sch-Val24} gives $3$ $\Aut(S)$-invariant unipotent characters. Further, Lemma \ref{lem:unipsfixed} gives that these characters must be $\sigma$-fixed. Thus, (i)(a) holds for $S$. 

We next show (ii).  Note that $3$ does not divide the size of $\tw{2}\type{B}_2(q)$, so we may assume $S\neq \tw{2}\type{B}_2(q)$.  
If $S\neq \PSL_2(q)$ and $S \neq \PSL^{\epsilon}_3(q)$ with $3\mid(q+\epsilon)$, then \cite[Prop.~5.8]{Gia-Riz-Sch-Val24} gives $3$ unipotent characters with distinct $3'$ degrees lying in the principal block, and these characters are $\sigma$-fixed by Lemma \ref{lem:unipsfixed}.

So, now assume $S= \PSL_2(q)$ or that $3 \mid (q+\epsilon)$ and  $S= \PSL^{\epsilon}_3(q)$. In these cases let  $G = \SL_2(q)$ and $G = \SL^{\epsilon}_3(q)$ respectively. We also let $\wt G =  \GL_2(q)$, resp. $\wt G = \GL^{\epsilon}_3(q)$ and $G^\ast = \PGL_2(q)$, resp. $ \PGL^{\epsilon}_3(q)$. Let $s\in  G^*$ be a semisimple element whose preimage under the canonical projection map $\pi :\tilde{G}\rightarrow G^*$  has nontrivial eigenvalues $(\zeta_3,\zeta_3^2)$ for a primitive third
root of unity $\zeta_3$. Consider a semisimple character  $\chi_s \in\mathcal{E}(G,s)$. Lemma \ref{lem:semisimplep} gives that $\chi_s$ deflates to a character in $\Irr_{\sigma}(B_0(S))$. Further, this character is height-zero by the degree properties of semisimple characters, which we see by computing the centralizer explicitly in $\wt{G}$.  Then the deflation of $\chi_s$ alongside the $2$ unipotent characters given in \cite[Prop.~5.8]{Gia-Riz-Sch-Val24} yield three $\sigma$-fixed characters in $B_0(S)$ with distinct $3'$-degrees, as desired. 
\end{proof}

Theorem \ref{thm:simple group conditions} now follows from Lemmas \ref{lem:sporalt} and \ref{lem:nondef}, recalling that part (iii) is \cite[Thm.~C(b)]{Riz-Sch-Val20}.

\subsection{Related Results}
We end by making some observations to give evidence of the analogue of Theorem \ref{thm:A} for arbitrary blocks. We first obtain the statement for groups of Lie type defined in characteristic $3$ whose ambient algebraic group has connected center.

\begin{cor}\label{cor:defchararbblocks}
Let $G=\bG^F$ be a group of Lie type, where $\bG$ is a connected reductive group over $\overline{\mathbb{F}}_3$ such that $\Z(\bG)$ is connected. Let $B$ be a $3$-block of $G$ with defect group $D$. Then $[D:\Phi(D)]=9$ if $|\Irr_{0,\sigma}(B)|\in\{6,9\}$.
\end{cor}
\begin{proof}

 By a result of Dagger and Humphreys, the only blocks with positive defect have maximal defect, and these are indexed by the characters of $\zent{{G}}$ (see \cite{humphreys}). 
 First assume $\bG$ has no component of type $\type{G}_2$. Then any block $B$ with positive defect satisfies $\Irr_{0}(B)=\Irr_{3'}(B)$ is comprised of semisimple characters in $B$, using the degree properties of Jordan decomposition and that nontrivial unipotent characters are divisible by $3$ by \cite[Thm.~6.8]{Malle07}. 

Since $\Z(\bG)$ is connected and the semisimple elements of ${G}^\ast$ have order prime to $3$ and $\sigma$ fixes $3'$-roots of unity, we then have $\Irr_0(B)=\Irr_{0,\sigma}(B)$ using the main result of \cite{SV20}.  
Assume that $|\Irr_{0,\sigma}(B)|\in\{6,9\}$ and let $B$ be the block indexed by the character $\theta\in\Irr(\Z(G))$.
Now, the number of semisimple characters in $\Irr(B)$ is the number of $G^\ast$-classes of semisimple elements $s\in G^\ast$ such that the characters in $\mathcal{E}(G, s)$ lie above $\theta$. But as noted in \cite[Prop.~2.7(iii)]{SFT23} (see also the proofs of \cite[Lem.~4.4]{Nav-Tiep16}, \cite[Prop.~3.7]{MNST}) this is the number of semisimple classes in the corresponding coset in $G^\ast/\mathbf{O}^{3'}(G^\ast)$. Now, for each semisimple $s\in G^\ast$, there is some $T^\ast:=(\bg{T}^\ast)^F$ containing $s$, where $\bg{T}^\ast\leq \bG^\ast$ is some $F$-stable maximal torus. By \cite[Rem.~1.5.13(b, c)]{GM20}, $T^\ast$ induces all cosets of $G^\ast/\mathbf{O}^{3'}(G^\ast)$.  Then by \cite[Cor.~2.4]{FG12}, we see each such coset contains the same number of semisimple classes. This means that $|\Irr_0(B)|=|\Irr_0(B_0(G))|$. Recalling that $\Irr_0(B_0(G))=\Irr_{0,\sigma}(B_0(G))$ as before, we therefore have $|\Irr_{0,\sigma}(B_0(G))|\in\{6,9\}$. Then $[P:\Phi(P)]=9$ by Theorem \ref{thm:A}. 

Finally, suppose that $\bG$ has components of type $\type{G}_2$. Here by \cite[Thm.~6.8]{Malle07}, the group $\type{G}_2(q)$ (where $q$ is a power of $3$) has 7 unipotent characters in $\Irr_{3'}(\type{G}_2(q))$. The unipotent characters of $\type{G}_2(q)$ are all $\sigma$-stable, using \cite[Table 1 and Prop.~5.5]{Geck03}. Note that these 7 unipotent characters then lie in $\Irr_{0,\sigma}(B_0(\type{G}_2(q)))$ since $\Z(\type{G}_2)$ is connected, and hence there is a unique block of positive defect from the first paragraph. Hence the statement is trivially satisfied if $\bG$ is simple. 

So, assume that $\bG$ is not simple. Since we have assumed $\Z(\bG)$ is connected, we have by \cite[Lem.~1.7.7]{GM20} that $\bG^F=[\bG, \bG]^F\bg{T}^F$ for any $F$-stable maximal torus $\bg{T}$ of $\bG$. Then $G$ is a quotient of the direct product of $[\bG, \bG]^F$ with an abelian group $\bg{T}^F$ of order prime to $3$, and hence we may without loss replace $\bG$ with $[\bG, \bG]$, and therefore assume that $\bG$ is semisimple. (Indeed, recall that $\sigma$ will fix any character in $\Irr(\bg{T}^F)$, each of which lie in their own block.)  Then $\bG=\prod_{i=1}^n \bG_i$ for some simple reductive groups $\bG_i$ which commute pairwise. We may assume that $\bG_1,\ldots,\bG_k$ are the components of type $\type{G}_2$ and write $\bg{H}$ for the product of these. Since each $\bG_i$ for $1\leq i\leq k$ has trivial center (as the only simple group of type $\type{G}_2$ is the adjoint group), it follows that $\bg{H}$ is a direct product and $\bG=\bg{H}\times \bg{H}'$ where $\bg{H}'$ is the product of simple components not of type $\type{G}_2$. In particular, by \cite[Cor.~1.5.16]{GM20}, $G$ has a direct factor of the form $\type{G}_2(q)$ with $q$ some power of $3$, and the result follows from the previous paragraph.
\end{proof}

The next observation follows from Theorem \ref{thm:A} and the fact that blocks of general unitary and linear groups behave similarly to certain principal blocks.
\begin{cor}
Let $\wt{G}=\GL_n^\epsilon(q)$ and let $B$ be a $3$-block of $\wt{G}$ with defect group $D$. Then $[D:\Phi(D)]=9$ if  
$|\Irr_{0,\sigma}(B)|\in\{6,9\}$.  
\end{cor}
\begin{proof}
If $q$ is a power of $3$, this is from Corollary \ref{cor:defchararbblocks}.
So now assume that $q$ is a power of some prime $p\neq 3$. Let $B$ be a block in $\mathcal{E}(\wt{G}, s)$ for some semisimple $3'$-element $s\in \wt{G}^\ast$.  Now, by \cite[Thm.~(7A)]{FS82}, Jordan decomposition yields a bijection between $\Irr_0(B)$ and $\Irr_0(b)$, where $b$ is some unipotent block of $C:=\cent{\wt{G}^\ast}{s}$. Here $B$ and $b$ share a defect group $D$. This bijection is further $\sigma$-equivariant by \cite{SV20}. 

Now, let $e$ be the order of $\epsilon q$ modulo $3$, so that $e\in\{1,2\}$. If $e=1$, then in fact the only unipotent block of $C$ is the principal block, so the result now follows from Theorem \ref{thm:A}.

So, suppose $e=2$. We may write $C:=\prod C_i$ as the product of lower-rank general linear or unitary groups $C_i$, depending on the invariant factors of $s$. Say $C_i:=\GL_{m_i}(\eta_i q_i)$ with $q_i$ some power of $q$, $\eta_i\in\{\pm1\}$, and $m_i$ a positive integer. Let $e_i$ be the order of $\eta_i q_i$ modulo $3$. Then the results of \cite{MO83} yield that $D$ is isomorphic to a Sylow $3$-subgroup of $C':=\prod \GL_{e_iw_i}(\eta_i q_i)$ for some positive integer $w_i$ and further there is a bijection between $\Irr_0(b)$ and $\Irr_0(B_0(C'))$. This is further $\sigma$-equivariant, which can be seen from the construction in \cite[P.209-210]{MO83} and using \cite{SV20} to see the action of $\sigma$ on characters in $\mathcal{E}(C_i, t)$ depend only on $s_i$ and $t$, and these will have the same eigenvalues as the corresponding embeddings in $C_i'$.
Hence, the result again follows from Theorem \ref{thm:A}.
\end{proof}

\begin{thebibliography}{MMSFV26}

\bibitem[Alp76]{alperin76}
J.~L. Alperin.
\newblock Isomorphic blocks.
\newblock {\em J. Algebra}, 43(2):694--698, 1976.

\bibitem[BP80]{Bro-Pui80I}
M.~Brou\'e and L.~Puig.
\newblock A {F}robenius theorem for blocks.
\newblock {\em Invent. Math.}, 56(2):117--128, 1980.

\bibitem[Bra68]{Bra68}
R.~Brauer.
\newblock On blocks and sections in finite groups. {II}.
\newblock {\em Amer. J. Math.}, 90:895--925, 1968.

\bibitem[Bra76]{Bra76}
R.~Brauer.
\newblock On finite groups with cyclic {S}ylow subgroups. {I}.
\newblock {\em J. Algebra}, 40(2):556--584, 1976.

\bibitem[CE04]{CE04}
M.~Cabanes and M.~Enguehard.
\newblock {\em Representation theory of finite reductive groups}, volume~1 of
  {\em New Mathematical Monographs}.
\newblock Cambridge University Press, Cambridge, 2004.

\bibitem[Dad77]{dade77}
E.~C. Dade.
\newblock Remarks on isomorphic blocks.
\newblock {\em J. Algebra}, 45(1):254--258, 1977.


\bibitem[FS82]{FS82}
P.~Fong and B.~Srinivasan.
\newblock The blocks of finite general linear and unitary groups.
\newblock {\em Invent. Math.}, 69(1):109--153, 1982.

\bibitem[FG12]{FG12}
J.~Fulman and R.~Guralnick.
\newblock Bounds on the number and sizes of conjugacy classes in finite
  {C}hevalley groups with applications to derangements.
\newblock {\em Trans. Amer. Math. Soc.}, 364(6):3023--3070, 2012.



\bibitem[GAP24]{GAP}
The GAP~Group.
\newblock {\em {GAP -- Groups, Algorithms, and Programming, Version 4.13.0}},
  2024.

\bibitem[Gec03]{Geck03}
M.~Geck.
\newblock Character values, {S}chur indices and character sheaves.
\newblock {\em Represent. Theory}, 7:19--55, February 2003.



\bibitem[GM20]{GM20}
M.~Geck and G.~Malle.
\newblock {\em The character theory of finite groups of {L}ie type}, volume 187
  of {\em Cambridge Studies in Advanced Mathematics}.
\newblock Cambridge University Press, Cambridge, 2020.
\newblock A guided tour.

\bibitem[Gia25]{Gia25}
E.~Giannelli.
\newblock{{M}c{K}ay bijections and character degrees.} \newblock \emph{arXiv:2507.01730,}
\newblock 2025.



\bibitem[GRSS20]{Gia-Riz-Sam-Sch20}
E.~Giannelli, N.~Rizo, B.~Sambale, and A.~A. {Schaeffer Fry}.
\newblock Groups with few {$p'$}-character degrees in the principal block.
\newblock {\em Proc. Amer. Math. Soc.}, 148(11):4597--4614, 2020.

\bibitem[GRSV25]{Gia-Riz-Sch-Val24}
E.~Giannelli, N.~Rizo, A.~A. Schaeffer~Fry, and C.~Vallejo.
\newblock Characters and {S}ylow subgroup abelianization.
\newblock {\em J. Algebra}, 667:824--864, 2025.



\bibitem[GI83]{Glu-Isa83}
D.~Gluck and I.~M. Isaacs.
\newblock Tensor induction of generalized characters and permutation
  characters.
\newblock {\em Illinois J. Math.}, 27(3):514--518, 1983.

\bibitem[Her70]{Her70}
M. Herzog.
\newblock On a problem of E. Artin.
\newblock {\em J. Algebra}, 15:408--416, 1970.

\bibitem[His90]{hiss}
G.~Hiss.
\newblock Regular and semisimple blocks of finite reductive groups.
\newblock {\em J. London Math. Soc. (2)}, 41(1):63--68, 1990.

\bibitem[Hum71]{humphreys}
J.~E.~ Humphreys.
\newblock Defect groups for finite groups of {L}ie type.
\newblock {\em Math. Z.}, 119:149--152, 1971.

\bibitem[HS23]{HSF}
N.~N. Hung and A.~A. Schaeffer~Fry.
\newblock On {H}\'ethelyi-{K}\"ulshammer's conjecture for principal blocks.
\newblock {\em Algebra Number Theory}, 17(6):1127--1151, 2023.



\bibitem[IN02]{Isa-Nav02}
I.~M. Isaacs and G.~Navarro.
\newblock New refinements of the {M}c{K}ay conjecture for arbitrary finite
  groups.
\newblock {\em Ann. of Math. (2)}, 156(1):333--344, 2002.

\bibitem[Ket25]{eden}
E.~Ketchum.
\newblock Characters and the generation of {S}ylow 3-subgroups for almost
  simple groups.
\newblock {\em arXiv:2509.02854}, 2025.

\bibitem[LMNT25]{Lyo-Mar-Nav-Tie25}
R.~Lyons, J.~M. Mart\'inez, G.~Navarro, and P.~H. Tiep.
\newblock Principal blocks, irreducible restriction, fields and degrees.
\newblock \emph{arXiv:2507.22503,}
\newblock 2025.

\bibitem[Mal07]{Malle07}
G.~Malle.
\newblock Height 0 characters of finite groups of {L}ie type.
\newblock {\em Represent. Theory}, 11:192--220, 2007.

\bibitem[Mal08]{Mal08}
G.~Malle.
\newblock Extensions of unipotent characters and the inductive {M}c{K}ay
  condition.
\newblock {\em J. Algebra}, 320(7):2963--2980, 2008.

\bibitem[MNST24]{MNST}
G.~Malle, G.~Navarro, A.~A. Schaeffer~Fry, and P.~H. Tiep.
\newblock Brauer's height zero conjecture.
\newblock {\em Ann. of Math. (2)}, 200(2):557--608, 2024.

\bibitem[MT11]{MT11}
G.~Malle and D.~Testerman.
\newblock {\em Linear Algebraic Groups and Finite Groups of Lie Type}.
\newblock Cambridge Studies in Advanced Mathematics. Cambridge University
  Press, 2011.
  
\bibitem[MMSV26]{Mar-Mar-Sch-Val24}
A.~Mar\'oti, J.~M. Mart\'inez, A.~A. Schaeffer~Fry, and C.~Vallejo.
\newblock On almost {$p$}-rational characters in principal blocks.
\newblock {\em Publ. Mat.}, 70(1):55--78, 2026.



\bibitem[MO83]{MO83}
G.~O. Michler and J.~B. Olsson.
\newblock Character correspondences in finite general linear, unitary and
  symmetric groups.
\newblock {\em Math. Z.}, 184(2):203--233, 1983.

\bibitem[MS23]{Mor-Sam23}
A.~Moret\'{o} and B.~Sambale.
\newblock Groups with 2-generated {S}ylow subgroups and their character tables.
\newblock {\em Pacific J. Math.}, 323(2):337--358, 2023.



\bibitem[Mur98]{Mur98}
M.~Murai.
\newblock Blocks of factor groups and heights of characters.
\newblock {\em Osaka J. Math.}, 35(4):835--854, 1998.

\bibitem[Nav98]{Nav98}
G.~Navarro.
\newblock {\em Characters and blocks of finite groups}, volume 250 of {\em
  London Mathematical Society Lecture Note Series}.
\newblock Cambridge University Press, Cambridge, 1998.

\bibitem[Nav04]{Nav04}
G.~Navarro.
\newblock The {M}c{K}ay conjecture and {G}alois automorphisms.
\newblock {\em Ann. of Math. (2)}, 160(3):1129--1140, 2004.

\bibitem[Nav18]{Nav18}
G.~Navarro.
\newblock {\em Character theory and the {M}c{K}ay conjecture}, volume 175 of
  {\em Cambridge Studies in Advanced Mathematics}.
\newblock Cambridge University Press, Cambridge, 2018.

\bibitem[NRSV21]{Nav-Riz-Sch-Val20}
G.~Navarro, N.~Rizo, A.~A. {Schaeffer Fry}, and C.~Vallejo.
\newblock Characters and generation of {S}ylow 2-subgroups.
\newblock {\em Represent. Theory}, 25:142--165, 2021.

\bibitem[NT13]{Nav-Tiep16}
G.~Navarro and P.~H. Tiep.
\newblock Characters of relative {$p'$}-degree over normal subgroups.
\newblock {\em Ann. of Math. (2)}, 178(3):1135--1171, 2013.

\bibitem[NT16]{Nav-Tie16}
G.~Navarro and P.~H. Tiep.
\newblock Real groups and {S}ylow 2-subgroups.
\newblock {\em Adv. Math.}, 299:331--360, 2016.

\bibitem[NT19]{Nav-Tie19}
G.~Navarro and P.~H. Tiep.
\newblock Sylow subgroups, exponents, and character values.
\newblock {\em Trans. Amer. Math. Soc.}, 372(6):4263--4291, 2019.

\bibitem[NT21]{Nav-Tie21}
G.~Navarro and P.~H. Tiep.
\newblock The fields of values of characters of degree not divisible by {$p$}.
\newblock {\em Forum Math. Pi}, 9:Paper No. e2, 28, 2021.

\bibitem[Riz18]{Riz18}
N.~Rizo.
\newblock {$p$}-blocks relative to a character of a normal subgroup.
\newblock {\em J. Algebra}, 514:254--272, 2018.

\bibitem[RSV20]{Riz-Sch-Val20}
N.~Rizo, A.~A. {Schaeffer Fry}, and C.~Vallejo.
\newblock Galois action on the principal block and cyclic {S}ylow subgroups.
\newblock {\em Algebra Number Theory}, 14(7):1953--1979, 2020.

\bibitem[RS25]{RSF25}
L.~Ruhstorfer and A.~A. Schaeffer~Fry.
\newblock The {I}saacs--{N}avarro {G}alois conjecture.
\newblock {\em arXiv:2509.02300}, 2025.



\bibitem[ST23]{SFT23}
A.~A. Schaeffer~Fry and J.~Taylor.
\newblock Galois automorphisms and classical groups.
\newblock {\em Transform. Groups}, 28(1):439--486, 2023.

\bibitem[SV20]{SV20}
B.~Srinivasan and C.~R. Vinroot.
\newblock Galois group action and {J}ordan decomposition of characters of
  finite reductive groups with connected center.
\newblock {\em J. Algebra}, 558:708--727, 2020.

\bibitem[TZ04]{TZ04}
P.~H. Tiep and A.~E. Zalesskii.
\newblock Unipotent elements of finite groups of {L}ie type and realization
  fields of their complex representations.
\newblock {\em J. Algebra}, 271:327--390, 2004.

\bibitem[Va23]{Va23}
C. Vallejo~Rodr\'iguez.
\newblock  A lower bound on the number of generators of a defect group. 
\newblock {\em Vietnam J. Math.}, 51(3):571--576, 2023.

\bibitem[VLVL85]{Ver-Ver85}
A.~Vera~L\'{o}pez and J.~Vera~L\'{o}pez.
\newblock Classification of finite groups according to the number of conjugacy
  classes.
\newblock {\em Israel J. Math.}, 51(4):305--338, 1985.

\end{thebibliography}

\end{document}